\documentclass[10pt]{amsart}

\RequirePackage{amsthm,amsmath,amsfonts,amssymb}
\RequirePackage[numbers]{natbib}

\usepackage{graphicx} \usepackage{enumerate} \usepackage{multicol}
\usepackage{subcaption}
\usepackage{mathrsfs} \usepackage[all,cmtip]{xy}
\usepackage{hyperref}
\usepackage[color=blue!20,textsize=footnotesize]{todonotes}

\newcounter{dummy} \numberwithin{dummy}{section}
\newtheorem{theorem}[dummy]{Theorem}

\newtheorem{lemma}[dummy]{Lemma}
\newtheorem{definition}[dummy]{Definition}
\newtheorem{proposition}[dummy]{Proposition}
\theoremstyle{remark}
\newtheorem{remark}[dummy]{Remark}
\newtheorem{example}[dummy]{Example}

\newcommand{\calH}{\mathcal{H}}
\newcommand{\calE}{\mathcal{E}}
\newcommand{\calR}{\mathcal{R}}
\newcommand{\calV}{\mathcal{V}}
\newcommand{\calP}{\mathcal{P}}
\newcommand{\calM}{\mathcal{M}}

\newcommand{\bsigma}{\boldsymbol \sigma}
\newcommand{\bpi}{\boldsymbol \pi}
\newcommand{\bgamma}{\boldsymbol \gamma}
\newcommand{\bSigma}{\boldsymbol \Sigma}
\newcommand{\blambda}{\boldsymbol \lambda}
\newcommand{\brho}{\boldsymbol \rho}

\DeclareMathOperator{\OM}{OM}

\DeclareMathOperator{\id}{id}

\DeclareMathOperator{\pr}{pr}
\DeclareMathOperator{\rank}{rank}
\DeclareMathOperator{\spn}{span}
\DeclareMathOperator{\tr}{tr}

\DeclareMathOperator{\Ort}{O}
\DeclareMathOperator{\DIV}{div}

\DeclareMathOperator{\Ad}{Ad}
\DeclareMathOperator{\ad}{ad}
\DeclareMathOperator{\GL}{GL}

\newcommand{\Lie}{\mathrm{Lie}}

\DeclareMathOperator{\SO}{SO}

\newcommand{\ptr}{/ \! /}
\newcommand{\tptr}{/ \! \tilde /}
\newcommand{\cptr}{/ \! \check /}
\newcommand{\ve}{\varepsilon}
\DeclareMathOperator{\Dev}{Dev}

\title[Most probable paths for developed processes]{Most probable paths for developed processes}

\author[E.~Grong and S.~Sommer]{Erlend Grong and Stefan Sommer}

\address{Department of Mathematics,
University of Bergen}
\email{erlend.grong@uib.no}

\address{Department of Computer Science,
University of Copenhagen}
\email{sommer@di.ku.dk}



\thanks{
The first author is supported by the grant GeoProCo from the Trond Mohn Foundation - Grant TMS2021STG02 (GeoProCo). The second author is supported by the Villum Foundation Grants 22924 and 40582, and the Novo Nordisk Foundation grant NNF18OC0052000.
}

\begin{document}
\begin{abstract}Optimal paths for the classical Onsager-Machlup function determining most probable paths between points on a manifold are only explicitly identified for specific processes, for example the Riemannian Brownian motion. This leaves out large classes of manifold-valued processes such as processes with parallel transported non-trivial diffusion matrix, processes with rank-deficient generator and sub-Riemannian processes, and push-forwards to quotient spaces. In this paper, we construct a general approach to definition and identification of most probable paths by measuring the Onsager-Machlup function on the anti-development of such processes. The construction encompasses large classes of manifold-valued process and results in explicit equation systems for the paths that we denote \emph{development most probable paths}. We define and derive these results and apply them to several cases of stochastic processes on Lie groups, homogeneous spaces, and landmark spaces appearing in shape analysis.
\end{abstract}

\maketitle

\section{Introduction}
The Onsager-Machlup function \cite{onsagerFluctuationsIrreversibleProcesses1953,machlupFluctuationsIrreversibleProcess1953} considered as a Lagrangian allows to identify most probable paths between points on manifolds for certain diffusion processes. The function measures the probability that a process sojourns in $\varepsilon$-tubes around differentiable paths in the $\varepsilon\to 0$ limit. First covering Euclidean space Brownian motion, the Onsager-Machlup function has been derived for Brownian motions on Riemannian manifolds \cite{takahashiProbabilityFunctionalsOnsagermachlup1981a,watanabeStochasticDifferentialEquations1981}
with extension to, e.g., time-inhomogeneous diffusions~\cite{CP14} and pinned-diffusions \cite{haraLagrangianPinnedDiffusion1996}. Most probable paths and the Onsager-Machlup function have also been treated in infinite dimensional manifolds \cite{demboOnsagerMachlupFunctionalsMaximum1991,capitaineOnsagerMachlupFunctionalElliptic2000, bardinaOnsagerMachlupFunctional2003,grongMostProbablePaths2022}. Still, only limited classes of manifold-valued processes are covered in the literature. Examples of processes that are not covered include processes defined by parallel transported anisotropic Brownian motion, sub-Riemannian processes, push-forwards through a projection from, for example, a Lie group to a homogeneous space, and many infinite dimensional non-linear shape processes.

Large classes of manifold-valued processes can be constructed by development of Euclidean semi-martingales to the manifold. Inspired by the constructions in \cite{sommerModellingAnisotropicCovariance2017,grongMostProbablePaths2022} that define most probable paths for anisotropic (non-Markovian) diffusion processes using the Onsager-Machlup function applied to the anti-development of the processes, we here construct a general approach to defining and identifying most probable paths for processes that are constructed by development. We call such curves \emph{developed most probable paths} or dMPPs for short, which are paths connecting two given points $x,y \in M$ whose anti-development is the most probable. This terminology allows us to derive differential equation systems for the resulting paths that both identify the set of dMPPs and allows numerical integration. While not capturing the volume distortion of the development map (see Section~\ref{sec:TrueMPPs} describing the usual definition of most probable paths and also Remark~\ref{re:distortion}), we argue that such curves are both intrinsically interesting and that they are relevant when the criteria for optimality relate to the Euclidean process that drives the developed process. Particularly, the anti-development can be seen as the processes that \emph{drives} the development, and the developed most probable paths thus measures path probability on the driving processes.

In \cite{sommerModellingAnisotropicCovariance2017,grongMostProbablePaths2022}, the dMPPs were identified for Brownian motions mapped through the frame bundle to the manifold. In this paper, we generalize beyond this specific setting to cover developed processes that include rank-deficient and time-varying diffusion field, and added drift on the manifold. We exemplify the results in several cases that are important in applications, including invariant processes on Lie groups and their push-forwards to homogeneous spaces, and, in the infinite dimensional case, paths of diffeomorphisms between landmark configurations as often encountered in shape analysis \cite{younesShapesDiffeomorphisms2010}. Stochastic shape analysis has applications in medical imaging and biology, and the presented results allows to find the most probable transformations between, e.g., a species shape and its immediate evolutionary ancestor compared to the geodesic transformations that are currently most used in shape analysis and where stochasticity is not taken into account.

\subsection{Related work}
Since the original work by Onsager and Machlup, \cite{onsagerFluctuationsIrreversibleProcesses1953,machlupFluctuationsIrreversibleProcess1953}, the Onsager-Machlup function and most probable paths have been the subject of many works including interpretation as a diffusion process Lagrangian \cite{bachOnsagerMachlupFunctionLagrangian1978}, extensions to manifolds
\cite{takahashiProbabilityFunctionalsOnsagermachlup1981a,watanabeStochasticDifferentialEquations1981,fujitaOnsagerMachlupFunctionDiffusion1982},
to time-inhomogeneous diffusions~\cite{CP14}, and as pinned-diffusions \cite{haraLagrangianPinnedDiffusion1996}. Extensions to infinite dimensional include \cite{demboOnsagerMachlupFunctionalsMaximum1991,capitaineOnsagerMachlupFunctionalElliptic2000, bardinaOnsagerMachlupFunctional2003,grongMostProbablePaths2022}. These works generally cover elliptic diffusion processes, and thus do not apply to the non-Markovian and sub-elliptic processes that appear from natural operations in differential geometry such as projection to a quotient space, or if a process driven by finite dimensional noise takes values in an infinite dimensional space. It is the purpose of this paper to handle these cases, albeit with the different approach of measuring path probability on the anti-development of the processes.

\subsection{Outline}
In Section~\ref{sec:context}, we set the context, survey the fundamental constructions used in the paper including development and anti-development of processes and prove the first result on variations of curves and development. Section~\ref{sec:eqs} contains the main results of the paper with the general equations for developed most probable paths. We treat multiple particular cases in Section~\ref{sec:examples}. In Section~\ref{sec:submersions}, we cover paths that result as projections through surjective submersions and where the dynamics can be lifted to the top space, e.g., homogeneous spaces. Section~\ref{sec:diffeos} covers Sobolev diffeomorphisms and actions on those on landmark spaces as appearing in shape analysis. We emphasize the fact that we can consider the landmark spaces as the base space of a submersion from the Lie group of Sobolev diffeomrophisms, so part of the motivation for studying Lie groups in Section~\ref{sec:examples} and submersions in Section~\ref{sec:submersions} is their eventual application in Section~\ref{sec:diffeos}. We end the paper with numerical examples on several manifolds in Section~\ref{sec:numerics}, which include Lie groups and landmarks. 

\section{Affine transport and development}
\label{sec:context}
In this section, we will make precise the context of developed processes, state general assumptions and derive first results linking the variation of curves with variations of their development. This provides the foundation of the identification of development most probable paths in the following sections.

\subsection{Affine transport}
Let $M$ be a manifold of dimension $d$ with a connection $\nabla$ and $o \in M$ a chosen point. Affine connections give a correspondence between curves in $T_oM$ starting at $0$ and curves in $M$ starting at $o$. For a smooth curve $[0,T] \to T_oM$, $t \mapsto \omega_t$ with $\omega_0 = 0$, its \emph{development} $\gamma:[0,T] \to M$ is defined as the solution to
$$\dot \gamma_t = \ptr_t \dot \omega_t, \qquad \gamma_0 =o,$$
where $\ptr_t: T_{o} M \to T_{\gamma_t}M$ denotes parallel transport along $\gamma_t$ with respect to $\nabla$. The curve $\omega_t$ is denoted the \emph{anti-development} of $\gamma_t$. Any curve $\gamma_t$ has an anti-development. Conversely, any curve $\omega_t$ has a development if $\nabla$ is \emph{complete}. For the remainder of the paper, we will assume that all connections are complete. For modifications when $\nabla$ is not complete, see \cite[Remark~2.2]{CGT21}.

We will also consider an affine analogue of the concept of development. Let $u_t$ be a time-dependent vector field. For a smooth curve $\omega_t$ starting at~$0$, we define \emph{the affine development} of $(\nabla, u_t)$ as the solution of
$$\dot \gamma_t =\ptr_t \dot \omega_t + u_t(\gamma_t), \qquad \gamma_0 = o.$$
Again, this exists for any smooth $\omega_t$ if $\nabla$ and $u_t$ are complete. We write $\gamma_t= \Dev_{u}(\omega)_t$ for the solution.

We will now produce a formula for the derivative of this affine development map, generalizing the result in \cite[Lemma~2.1]{CGT21} for the derivative of the development map. We will first introduce the following notation. Let $\tau_t$ be a time-dependent $\binom{i}{j}$-tensor and $\gamma_t$ be a curve in $M$ with $\gamma_0 = o$. We then define $\tau_{t,\ptr_t} \in x(T_o^*M)^{\otimes i} \otimes (T_oM)^{\otimes j} $ as the corresponding linear map obtained by parallel transporting all elements to $o =\gamma_0 \in M$. For example, if $T$ and $R$ are respectively the torsion and the curvature of $\nabla$, then
$$T_{\ptr_t}(v,w) = \ptr_t^{-1} T(\ptr_t v, \ptr_t w), \qquad R_{\ptr_t}(v,w) = \ptr_t^{-1} R(\ptr_t v, \ptr_t w) \ptr_t,$$
for any $v,w \in T_oM$, while  $u_{t,\ptr_t} = \ptr_t^{-1} u_t$ for a time-dependent vector field $u_t$.
Using this notation, we now have the following result for first-order variations of the development.
\begin{lemma} \label{lemma:AffineDev}
Let $k_t$ be a smooth map in $T_o M$ satisfying $k_0 = 0$. 
Then
$$\frac{d}{ds} \Dev_u(\omega+sk)_t |_{s=0}= \ptr_t v_t,$$
where $v_t$ is a curve in $T_oM$ with $v_0 = 0$ solving
\begin{align} \label{ktov}
k_t & = v_t  - \int_0^t \ptr_s^{-1} \nabla_{\ptr_s v_s} u_s ds -\int_0^t T_{\ptr_s}(  \dot \omega_s + u_{s,\ptr_s},v_s) ds \\ \nonumber
& \qquad - \int_0^t \int_0^s R_{\ptr_r}(\dot \omega_r + u_{r,\ptr_r}, v_r) \dot \omega_s \, dr ds .
\end{align}
\end{lemma}

\begin{proof}
Write $\gamma_{t}^s = \Dev_u(\omega + sk)_t$. We let $\frac{D}{dt}$ and $\frac{D}{ds}$ denote the covariant derivative along $\gamma_t^s$ in respectively the $t$ and $s$-directions relative to $\nabla$.
Construct a basis $e_{1,t}^s, \dots, e^s_{d,t}$ of vector fields along $\gamma_{t}^s$ uniquely determined by their initial conditions and $\frac{D}{dt} e_{j,t}^s=0$. In other words, we can choose an arbitrary basis of vector fields $e^s_{1,0}, \dots, e_{d,0}^s$ along $s \mapsto \gamma_0^s$ and parallel transport this basis in the $t$-direction. We simplify the notation by writing $e^0_{j,t} = e_{j,t}$. We use the standard formulas,
\begin{align*} T(\partial_s \gamma_t^s, \partial_t \gamma_t^s) &= \frac{D}{ds} \partial_t \gamma_t^s - \frac{D}{dt} \partial_s \gamma_t^s, \\
R(\partial_t \gamma_t^s, \partial_s \gamma_t^s) e_{j,t}^s & = \frac{D}{dt} \frac{D}{ds} e_{j,t}^s - \frac{D}{ds} \frac{D}{dt} e_{j,t}^s
= \frac{D}{dt} \frac{D}{ds} e_{j,t}^s.
\end{align*}
For the curves $T_oM$, if we write them in our chosen basis as $k_t = \sum_{j=1}^d k_t^j e_{j,0}$, $\omega_t = \sum_{j=1}^d \omega_t^j e_{j,0}$ and $\partial_s \gamma_t |_{s=0} = \ptr_t v_t$, then
\begin{align*}
& \dot v_t = \ptr_t^{-1} \frac{D}{dt} \partial_s \gamma_t^s |_{s=0} = \ptr_t^{-1} \left. \left( \frac{D}{ds} \partial_t \gamma_t^s - T(\partial_s \gamma_t^s, \partial_t \gamma_t^s)  \right) \right|_{s=0} \\
& = \ptr_t^{-1} \left( \sum_{i=1}^d \dot k^i_t e_{j,t} + \sum_{i=1}^n \dot \omega^j_t \frac{D}{ds} e_{j,t}^s|_{s=0} + ( \nabla_{\ptr_t v_t} u_t)(\gamma_t) - T(\ptr_t v_t, \ptr_t \omega_t + u_t(\gamma_t))  \right) \\
& = \dot k_t  + \sum_{i=1}^n \dot \omega^j_t \frac{D}{ds} e_{j,t}^s|_{s=0} + \ptr_{t}^{-1}(\nabla_{\ptr_t v_t} u_t)(\gamma_t) - T_{\ptr_t}( v_t, \omega_t + \ptr_t^{-1} u_t(\gamma_t)) .
\end{align*}
Furthermore,
\begin{align*}
\frac{d}{dt} \ptr_t^{-1} \frac{D}{ds} e_{j,t}^s|_{s=0} & = \ptr_t^{-1} \frac{D}{dt} \frac{D}{ds} e_{j,t}^s|_{s=0} = \ptr_t^{-1} R(\partial_t \gamma_t^s, \partial_s \gamma_t^s) e_{j,t}^s|_{s=0} \\
& =  R_{\ptr_t}(\dot \omega_t + \ptr_t^{-1} u_t(\gamma_t), v_t) e_j .
\end{align*}
Finally, since $s \to \gamma_0^s =o$ is constant, we have $D_s e_{j,t}^s |_{s=0,t=0} =0$ by definition, and hence
$$\frac{D}{ds} e_{j,t}^s |_{s=0} = \ptr_t \int_0^t  R_{\ptr_r}(\dot \omega_r + \ptr_r^{-1} u_r(\gamma_r), v_r) e_j dr.$$
Using that also $v_0 = 0$, the result follows.
\end{proof}

\subsection{Most probable paths in Euclidean spaces and Riemannian manifolds} \label{sec:TrueMPPs}
Before considering the notion of development most probable paths that we investigate in this paper, we first recap the standard notions of most probable curves in $\mathbb{R}^d$ and on Riemannian manifolds. Let $B_t$ be a Brownian motion in $\mathbb{R}^d$.
Let $C^1([0,T],\mathbb{R}^d;0)$ be the space of $C^1$ curves $t \mapsto z_t$ such that $z_0 = 0$. We then say that $B_t$ has Onsager-Machlup function $\OM:C^1([0,T],\mathbb{R}^d;0) \to \mathbb{R}$, if 
\begin{equation} \label{OMlimit}\lim_{\ve \downarrow 0} \frac{P(\|B_t -z_{1,t}\|< \ve\ \forall\,t \in [0,T])}{P(\|B_t -z_{2,t}\|< \ve\ \forall\,t \in [0,T])} = \exp\left( -\OM(z_{1}) + \OM(z_2) \right)\end{equation}
for all $z_1,z_2\in C^1([0,T],\mathbb{R}^d;0)$.
In this case where $B_t$ is a Brownian motion, the Onsager-Machlup function \begin{equation} \label{OMRd} \OM(z) = \OM_{\mathbb{R}^d}(z) = \frac{1}{2} \int_0^T \langle \dot z_t, \dot z_t \rangle \, dt.\end{equation}
Hence, we can consider a path $t \mapsto z_t$ \emph{as most probable} if it minimizes $\OM$ among all the paths with the same fixed end points, which for $\mathbb{R}^d$ are straight lines.

Let now $(M,g)$ be a Riemannian manifold with a vector field $u$, and fix an initial point $o \in M$. We emphasize that $u$ is a usual vector field in this example without any time dependence. Let $B_t$ be a Brownian motion in the inner product space $T_oM$. We consider the solution of the manifold-valued SDE,
$$dx_t = \ptr_t \circ dB_t + u(x_t) dt, \qquad x_0 = o.$$
Let $\gamma_t$ be a curve in $M$ starting at $o$, and consider the probability of $d_g(x_t,\gamma_t)< \ve$ for every $t \in [0,T]$ where $d_g$ is the Riemannian distance. If we consider the asymptotic limit as $\ve \to 0$ as in \eqref{OMlimit}, the Onsager-Machlup function is
\begin{equation} \label{OMRie} \OM_{M,g,u}(\gamma_t) = \int_0^T \left( \frac{1}{2} \| \dot \gamma_t -u(\gamma_t)\|_g^2 + \DIV(u)(\gamma_t) - \frac{1}{12} S(\gamma_t) \right) dt.\end{equation}
with $\DIV$ being the divergence and $S$ the scalar curvature. See \cite{fujitaOnsagerMachlupFunctionDiffusion1982} for details. Most probable paths are then critical points of \eqref{OMRie} with fixed endpoints. For explicit formulas for critical points of $\OM_{M,g,u}$, see Remark~\ref{re:distortion}.

\subsection{Diffusion processes with drift and non-zero covariance} \label{sec:dMPP}
We will now consider a more general class of processes than those in Section~\ref{sec:TrueMPPs}. Let $M$ be a given $d$-dimensional connected manifold. \emph{A sub-Riemannian structure} $(E,g)$ is a subbundle $E \subseteq TM$ along with a smoothly varying inner product $g = \langle \cdot , \cdot \rangle_g$ defined only on $E$. We include here the possibility that $E = TM$, which makes $(E,g)$ a Riemannian structure. An affine connection $\nabla$ is said to be \emph{compatible} with $(E,g)$ if it takes orthogonal bases of $E$ to orthogonal bases under parallel transport.

Let $o \in M$ be a given base point on the sub-Riemannian manifold $(M, E, g)$. For a given $T >0$, define \emph{the Cameron Martin space} $\mathbb{H}^T(E_o)$ as the space of all absolutely continuous curves in $E_o$ starting at $0$ with an $L^2$-derivative with respect to our chosen inner product. This is a Hilbert space with product
$$\langle \omega, k\rangle_{\mathbb{H}} = \int_0^T \langle \dot \omega_t, \dot k_t \rangle_g \, dt.$$
If $\nabla$ is a given connection and $u_t$ a given time-dependent vector field, we notice that the concept of affine development $\Dev_u(\omega)$ is well defined for curves $\omega \in \mathbb{H}^T(E_o)$, and we can still apply Lemma~\ref{lemma:AffineDev} with its proof. See Appendix~\ref{sec:AbsCont} for details.

We consider parallel transport and development with respect to a connection $\nabla$ compatible with the sub-Riemannian structure $(E,g)$. Let $\Sigma_t: E_o \to E_o$ be a linear map with a smooth dependence on $t \in \mathbb{R}_{\geq 0}$. We also assume that $\Sigma_t$ is invertible, positive definite and symmetric with respect $\langle \cdot, \cdot \rangle_{g(o)}$. Finally, we assume that it the unique positive square root $\Sigma_t^{1/2}$ of $\Sigma_t$ has smooth dependence on $t$ as well. For the remainder of the paper, we will always have these assumptions for $\Sigma_t$.

We consider the solution of $x_0 = o$ and
\begin{equation} \label{xtFirsteq} dx_t = \ptr_t \Sigma_t^{1/2} \circ dB_t + u_t(x_t) \,dt,\end{equation}
where $B_t$ is a standard Brownian motion in the inner product space $(E_{o}, \langle \cdot, \cdot \rangle_{g(o)})$ and $\ptr_t: T_o M \to T_{x_t}M$ denotes parallel transport along $x_t$ with respect to $\nabla$.
 For more details finding stochastic processes from parallel transport, we refer to, e.g., \cite{elworthy1992stochastic,Hsu02}. Notice if $\rank E = J$, then there is always some orthogonal basis $\sigma_{1,t}(o),\dots,\sigma_{J,t}(o)$ smoothly on~$t$ such that $\Sigma_t(v) = \sum_{j=1}^J \langle \sigma_{j,t}(o) , v \rangle \sigma_{j,t}(o)$, $v \in E_o$, which allows us to alternatively write \eqref{xtFirsteq} as
$$dx_t = \ptr_t \sum_{j=1}^J \sigma_{j,t}(o) \circ d\hat B^j_t + u_t(x_t) \, dt,$$
where $(\hat B_t^1, \dots, \hat B_t^J)$ is a standard Brownian motion in $\mathbb{R}^J$.

To the best of our knowledge, there are no general expression for the Onsager-Machlup function for $x_t$ even in the case when $u_t =0$ and when $\Sigma_t = \Sigma$ is constant in time and of full rank, but not a constant times the identity. We emphasize that the previous mentioned case is non-Markov, see Remark~\ref{re:NonMarkov} for details. As such, with slight abuse of terminology, we introduce the following replacement. For any $y \in M$, we define $\mathbb{H}(y) = \mathbb{H}^T(E_o; \nabla, \Sigma_t, u_t,y)$ as the subset of all curves $\omega_t \in \mathbb{H}^T(E_o)$ such that $\gamma_t$ is the solution of
\begin{equation} \label{gammatFirsteq} \dot \gamma_t = \ptr_t \Sigma_t^{1/2} \dot \omega_t + u_t(\gamma_t),\qquad \gamma_0 = o,\end{equation}
and
\begin{equation} \label{DefPi}
\Pi(\omega) :=\gamma_T = y.
\end{equation}
This leads to the following definition.
\begin{definition}[Development most probable paths]
A curve $\gamma_t\in M$ satisfying \eqref{gammatFirsteq} and \eqref{DefPi} for some given $\omega$ is \emph{a development most probable path} (dMPP) from $o$ to $y$ if $\omega$ is a most probable path in $\mathbb{H}(y)$ with respect to $B_t$.
\end{definition}
In other words, $\gamma_t$ is a development most probable path if its anti-development
\begin{equation} \label{Anti-dev} \omega_t = \int_0^t \Sigma_s^{-1/2} \ptr_s^{-1}(\dot \gamma_s - u_s(\gamma_s) ) ds. \end{equation}
is in $\mathbb{H}^T(E)$ and is a critical point of
\begin{equation} \label{OMEo}
\OM_{E_o}(\omega)  = \frac{1}{2} \| \omega \|_{\mathbb{H}}= \frac{1}{2} \int_0^T \langle \dot \omega_t, \dot \omega_t \rangle_{g(o)} dt
\end{equation}
when restricted to $\mathbb{H}(y)$. We emphasize that contrary to MPPs for Euclidean Brownian motions, dMPPs are not developments of straight lines because we are fixing the end point in the manifold rather than in the tangent space. The dMPPs will sense the curvature of the manifold through development map and the endpoint condition.
See also Remarks~\ref{re:distortion} for differences between MPPs and dMPPs and the effect of $\Sigma_t^{1/2}$ in the definition of dMPPs.

We will give equations for developed most probable paths in Theorem~\ref{th:MPPu}. We will first, however, present the following useful lemma. Introduce a tensor~$K$ seen as a map $K: \wedge^2 E \otimes TM \to T^*M$ and defined by
\begin{equation} \label{defK} K(v_1, v_2, w_1)(w_2) = K(v_1\wedge v_2, w_1)(w_2) = \langle R(w_1, w_2) v_1 , v_2 \rangle_g.\end{equation}
The operator $K$ above can be considered as a dual to the curvature operator~$R$ of~$\nabla$, where the curvature is only applied to $E$. We have used that the connection $\nabla$ is compatible with the sub-Riemannian structure $(E,g)$ and hence the right-hand side of \eqref{defK} is skew-symmetric in~$v_1$ and~$v_2$. We will also need the map $\sharp^g: T^*M \to E$ defined such that 
$$\alpha(v) = \langle \sharp^g \alpha, v \rangle_g, \qquad \text{for any $\alpha \in T^*_xM$, $v \in E_x$, $x \in M$.} $$
Relative to this notation, we have the following result.
\begin{lemma} \label{lemma:LambdaToAlpha}
Write $\frac{D}{dt}$ for the covariant derivative along $\gamma_t$ with respect to $\nabla$. Let $\omega, k \in \mathbb{H}^T(E_o)$ be arbitrary and define $\gamma_t^s = \Dev_u(\omega+ sk)$ and $\frac{d}{ds} \gamma_t^s |_{s=0} = \ptr_t v_t$. Then for any absolutely continuous curve $\lambda_t$ in $T^*M$ along $\gamma_t$, we have
\begin{align*}
\int_0^T \lambda_t(\ptr_t \dot k_t) \, dt = \lambda_T(\ptr_T v_T) - \int_0^T \alpha_t(\ptr_t v_t) \,dt
\end{align*}
where
$$\alpha_t = \frac{D}{dt} \lambda_t + \lambda_t(\nabla u_t)+ \lambda_t(T(\dot \gamma_t, \cdot)) - K(\chi_t, \dot \gamma_t),$$
and $\chi_t$ is the two-vector field along $\gamma_t$ with values in $\wedge^2 E$ that is the solution of
$$\frac{D}{dt} \chi_t = (\dot \gamma_t - u_t(\gamma_t)) \wedge \sharp^g \lambda_t, \qquad \chi_T = 0.$$
\end{lemma}
\begin{proof}
We use equation \eqref{ktov}. Observe first that
\begin{align*}
    & \int_0^T \int_0^t \lambda_t \ptr_t R_{\ptr_r}(\dot \omega_r + u_{r,\ptr_t}(\gamma_r), v_r ) \dot \omega_t \, dr dt \\
    & = \int_0^T \int_t^T \lambda_r \ptr_r \ptr_t^{-1} R(\dot \gamma_t , \ptr_t v_t ) \ptr_t \ptr_r^{-1} (\dot \gamma_r - u_r(\gamma_r)) \, dr dt \\
    & =  - \int_0^T K(\chi_t, \dot \gamma_t) ( \ptr_t v_t )  dt,
\end{align*}
meaning that
\begin{align*} 
\int_0^T \lambda_t (\ptr_t \dot k_t) \, dt & = \int_0^T \lambda_t(\ptr_t \dot v_t) \, dt  - \int_0^T \lambda_t(\nabla_{\ptr_t v_t} u_t) dt \\ \nonumber
& \qquad -\int_0^T \lambda_t T(  \dot \gamma_t,\ptr_t v_t) dt  + \int_0^T K(\chi_t, \dot \gamma_t) ( \ptr_t v_t )  dt \\
& = \lambda_T(\ptr_T v_T) - \int_0^t \alpha_t(\ptr_t v_t),
\end{align*}
by using integration by parts.
\end{proof}

\begin{remark}
For a general curve $\gamma_t$ in $M$, we have no guarantee that the solution of \eqref{Anti-dev} will take values in $E_o$. As an alternative statement, we could instead consider $\mathbb{H}(y)$ be all curves in $H(T_oM)$ satisfying $\Pi(\omega) = y$, but letting $\langle v , v \rangle_{g(o)} = \infty$ whenever $v \in T_oM \setminus E_o$.
\end{remark}

\subsection{Singular curves}
If we consider the map $\Pi: \mathbb{H}^T(E_o) \to M$, $\omega \mapsto \gamma_T$ as defined in \eqref{DefPi}, this is a smooth map of Hilbert manifolds. As a consequence, if $\Pi(\omega) = y$ and $\omega$ is a regular point of $\Pi$, then $\mathbb{H}(y) = \Pi^{-1}(y)$ has the structure of a Hilbert manifold close to $\omega$ by the implicit function theorem. We say that $\omega$ is \emph{regular} if it is a regular point of $\Pi$ and we will otherwise call it \emph{singular}. We say that a curve $\gamma$ is \emph{singular} if it is the solution of \eqref{gammatFirsteq} for a singular $\omega$.

We have the following criterion for singular curves.
\begin{lemma} \label{lemma:Singular}
The curve $\gamma$ is singular if and only if there is if a non-vanishing one-form $\lambda_t$ along $\gamma_t$ satisfying
\begin{equation} \label{singular} \left\{
\begin{aligned}
\dot \gamma_t - u_t(\gamma_t) \in E, \text{for almost every t}, \\
\frac{D}{dt} \lambda_t + \lambda_t \nabla u_t + \lambda_t T(\dot \gamma_t, \cdot ) =0, \\
\lambda_t|_E =0.
\end{aligned} \right.
\end{equation}
\end{lemma}

\begin{proof}
Recall that $\Pi$ sends a curve $\omega$ in $E_o$ to the final point of the solution of $\dot \gamma_t = \ptr_t \Sigma_t^{1/2} \dot \omega_t + u_t(\gamma_t)$, $\gamma_0 = o$. We can write this map as a composition of maps $\Pi = \tilde \Pi \circ I_{\Sigma}$, where
$$I_{\Sigma}(\omega) = \tilde \omega, \qquad \tilde \omega_t = \int_0^t \Sigma_s^{1/2} \dot{\omega}_s ds, \qquad \tilde \Pi(\tilde \omega) = \Dev_u(\tilde \omega).$$
Since the map $I_\Sigma$ is an invertible map of $\mathbb{H}^{T}(E_o)$, we have that $\omega$ is a singular point of $\Pi$ if and only if $\tilde \omega = I_\Sigma(\omega)$ is a singular point of $\tilde \Pi$. For the remainder of the proof, we will therefore describe the curves $\gamma = \Dev_u(\omega)$ such that $\omega \in \mathbb{H}^T(E_o)$ is a singular value of $\tilde \Pi$.

Let $\tilde \Pi_{*, \omega}$ denote its push-forward at $\omega$. For $\omega, k \in \mathbb{H}^T(E_o)$, let $v_t$ be defined as in \eqref{ktov}, which gives us the relation $\tilde \Pi_{*,\omega} k = \ptr_T v_T$ as in Lemma~\ref{lemma:AffineDev}. Hence, $\gamma$ is singular if for some non-zero $\lambda_T \in T_{\Pi(\omega)}^* M$, we have that
$$\lambda_T \ptr_T v_T = 0,$$
for all solutions $v_t$ of \eqref{ktov}.

Consider now $\lambda_t$ as the unique one-form along $\gamma_t$ with final value $\lambda_T$ that solves the equations
\begin{equation}  \label{Pre-MPP} \left\{ \begin{aligned}
  \frac{D}{dt} \lambda_t & =  - \lambda_t(\nabla u_t)- \lambda_t(T(\dot \gamma_t, \cdot)) + K(\chi_t, \gamma_t),  \\
  \frac{D}{dt} \chi_t & = (\gamma_t - u_t(\gamma_t)) \wedge \sharp^g \lambda_t, \qquad \chi_T = 0.
\end{aligned} \right.
\end{equation}
It then follows that
$$\lambda_T(\ptr_T v_T) = \int_0^T \lambda_t(\ptr_t \dot k_t)\, dt.$$
Since $\dot k_t$ can be an arbitrary $L^2$-function in $E_o$, we must hence have that $\lambda_t$ vanishes on $E_{\gamma_t} =\ptr_t E_o$. Finally, we use that $\sharp^g \lambda_t = 0$, giving us that $\chi_t =0$ and we have the result.
\end{proof}

\begin{remark} \label{re:NoSingular}
We mention in particular that if $E_o = T_oM$, $TM = E$, then a one-form cannot both be non-vanishing and vanish on $E$ and hence there are no singular curves.
\end{remark}

\section{Equations for developed most probable paths}
\label{sec:eqs}

We are now ready to present our main result. Let $x_t$ be the solution of \eqref{xtFirsteq}, i.e.,
\begin{equation} dx_t = \ptr_t \Sigma_t^{1/2} \circ dB_t + u(x_t) \,dt,\end{equation}
for a given $\Sigma_t: E_o \to E_o$ positive-definite symmetric map, vector field $u_t$ and parallel transport that preserve a sub-Riemannian structure $(E,g)$. In what follows, it will be practical to introduce the following abuse of notation, where we allow $\Sigma_t$ to act on elements in $E_{\gamma_t}$ by $\Sigma_t v = \ptr_t \Sigma_t \ptr_{t}^{-1}v$. With this terminology, we have the following equation for developed most probable paths.
\begin{theorem} \label{th:MPPu}
If $\gamma:[0,T] \to M$ is a developed most probable path of $x_t$ and not singular, then it is the solution of
\begin{equation}  \label{MPPu} \left\{ \begin{aligned}
 \sharp^g \lambda_t &= \Sigma_t^{-1} (\dot \gamma_t -u_t(\gamma_t)) \\
  \frac{D}{dt} \lambda_t & =  - \lambda_t(\nabla u_t)- \lambda_t(T(\dot \gamma_t, \cdot)) + K(\chi_t, \dot \gamma_t),  \\
  \frac{D}{dt} \chi_t & =  \frac{1}{2} \Sigma_t \sharp^g \lambda_t \wedge \sharp^g \lambda_t, \qquad \chi_T = 0.
\end{aligned} \right.
\end{equation}\end{theorem}
It follows that a dMPP has to satisfy either Lemma~\ref{lemma:Singular} or Theorem~\ref{th:MPPu} or both. Such a condition parallel the case for general length minimizers in sub-Riemannian manifolds, since such curves will have to be singular or solutions to a given Hamiltonian equation or both, see, e.g., \cite{montgomery2002tour} for details. Also see Section~\ref{sec:RSR} where these parallels are explored in more detail.
\begin{proof}
Let $\omega,\tilde \omega$ be curves in $E_o$ such that
$$\textstyle \dot\gamma_t = \ptr_t  \dot {\tilde \omega}_{t}  + u_t(\gamma_t) :=\ptr_t \Sigma_t^{1/2} \dot \omega_t  + u_t(\gamma_t), \qquad \tilde \omega_t  = I_\Sigma(\omega)_t = \int_0^t \Sigma_r^{1/2} \dot \omega_r \, dr,$$
and write $y = \gamma_T = \Pi(\omega) = \tilde \Pi(\tilde \omega)$. Define vector spaces
\begin{align*}
    V & = \left\{ \frac{d}{ds} \omega^s|_{s=0} \, : \text{$\omega^s$ is a variation of $\omega$, $\Pi(\omega^s) =y$} \right\}, \\
    \tilde V & = \left\{ \frac{d}{ds} \tilde \omega^s|_{s=0} \, : \text{$\tilde \omega^s$ is a variation of $\tilde \omega$, $\tilde \Pi(\tilde \omega^s) =y$} \right\},
\end{align*}
where $\tilde V = I_\Sigma(V)$. Remark that we assumed that $\omega$ and $\tilde \omega$ are regular points of $\Pi$ and $\tilde \Pi$, respectively, both $\Pi^{-1}(y)$ and $\tilde \Pi^{-1}(y)$ are submanifolds of $\mathbb{H}^T(E_o)$ close to the mentioned points of codimension $d = \dim M$ by the implicit function theorem. Along the curve, also introduce the following vector space of forms
$$\Lambda = \{\text{$\lambda_t$ is a solution of \eqref{Pre-MPP}}\},$$
which is a $d$-dimensional space by definition, as solutions are uniquely determined by the final condition $\lambda_T \in T_{\Pi(\omega)}M = T_{\tilde \Pi(\tilde \omega)} M$. Since $\gamma_t$ is not singular, $\lambda_t|_{E}$ is non-vanishing for every $t$ for any non-zero $\lambda_{\cdot} \in \Lambda$ by Lemma~\ref{lemma:Singular}. As a consequence
$$\textstyle \check{\Lambda} = \{ \int_0^t \ptr_r^{-1}\sharp^g \lambda_r  \, dr  \}\subseteq \mathbb{H}^T(E_o).$$
is also a $d$-dimensional subspace. By Lemma~\ref{lemma:LambdaToAlpha}
$$\textstyle \langle h, k \rangle_{\mathbb{H}} = \int_0^T \lambda_t(\ptr_t \dot k_t) = 0, \qquad \text{for $h = \int_0^{\cdot} \ptr_r^{-1} \sharp^g \lambda_r\, dr \in \check{\Lambda}$, and for any $k \in \tilde V$}.$$ 
Since $\check{\Lambda}$ is $d$-dimensional, it follows that $\mathbb{H}^T(E_o) = \check{\Lambda} \oplus_{\perp} \tilde V$, and furthermore
$$\textstyle \mathbb{H}^T(E_o) = \hat \Lambda \oplus_{\perp} V, \qquad \hat \Lambda = I_{\Sigma} (\check{\Lambda}).$$
Here we have used that $V = I_{\Sigma}^{-1} \tilde V$ and that $\langle I_{\Sigma}h, I_{\Sigma}^{-1} k \rangle_{\mathbb{H}} = \langle h, k \rangle_{\mathbb{H}}$ for $h,k \in \mathbb{H}(E_o)$.
If $\gamma$ is a most probable path, then we must have
$$0 = \frac{d}{ds} \OM_{E_o}(\omega^s) = \int_0^T \langle \dot \omega_t, \dot k_t \rangle \, dt = \langle \omega, k \rangle_{\mathbb{H}},$$
for any $k \in V$, meaning that $\omega \in \hat \Lambda$. In conclusion, for some $\lambda_t$ in $\Lambda$, we must have that
$$\ptr_t^{-1} \sharp^g \lambda_t = \Sigma^{-1/2}_t \dot \omega_t = \Sigma_t^{-1} \dot \omega_{2,t} = \Sigma_t^{-1} \ptr_t (\dot\gamma_t - u_t(\gamma_t)).$$
The result follows.
\end{proof}

\begin{remark}
We cannot conclude \eqref{MPPu} directly from Lemma~\ref{lemma:LambdaToAlpha} without the assumption that the curve is not singular, since the vector fields $t \mapsto \ptr_t v_t$ resulting from affine development will not be all vector fields when $E$ is strictly contained in $TM$. In fact, one can verify that if $\pr_E: TM \to E$ is any projection, then $v_t$ is uniquely determined by $\pr_E v_t$.
\end{remark}

\begin{remark} \label{re:Sigma} Let $\Pi, \tilde \Pi: \mathbb{H}^T(E_o) \to M$ and $I_\Sigma: \mathbb{H}^T(E_o) \to \mathbb{H}^T(E_o)$ be as in the proof of Lemma~\ref{lemma:Singular}.
If we consider $\dot \gamma_t = \ptr_t \Sigma_t^{1/2} \dot \omega_t + u_t(\gamma_t) = \ptr_t \dot {\tilde \omega}_{t} + u_t(\gamma_t)$, we showed in the proof of Lemma~\ref{lemma:Singular} that mapping $\omega$ to $\gamma$ or mapping $\tilde \omega = I_{\Sigma}(\omega)$ to $\gamma$ gives us the same singular values. We also obtain the same development most-probable paths if we optimize the probability of $\int_0^t \Sigma_s^{1/2} \circ dB_s$ sojourning around $\tilde \omega_{t} = \int_0^t \Sigma_s^{1/2} \dot \omega_s ds$. 

To see this, we can consider $\int_0^t \Sigma_s^{1/2} \circ dB_s$ as the Brownian motion of a time-dependent flat metric $G(t) = \langle \Sigma(t)^{-1} v, w \rangle_{g_o}$ on $E_o$. By \cite{CP14}, the Onsager-Machlup functional of $\int_0^t \Sigma_s dB_s$ on $E_o$ is then given by
$$\tilde \omega \mapsto \int_0^{T} \left(\| \dot { \tilde \omega}_{t} \|^2_{G(t)} + \tr_{G(t)} \dot G(t) \right) dt = \int_0^{T} \left(\| \Sigma^{-1/2}_t \dot {\tilde \omega}_{t} \|^2_{g(o)} + \tr_{g(o)} \Sigma_t \dot \Sigma_t^{-1} \right) dt.$$
Since $\|\Sigma^{-1/2}_t \dot {\tilde \omega}_{t} \|_{g(o)} = \| \dot \omega_{t} \|_{g(o)}$, minimizing the above functional over $\tilde \Pi^{-1}(y)$ gives the same result as minimizing $\OM$ as in \eqref{OMEo} over $\Pi^{-1}(y)$.
\end{remark}

\begin{remark} \label{re:distortion}
For the special case $E_o = T_o M$, i.e., the case of processes in a Riemannian manifold, let us compare the results of Theorem~\ref{th:LMDown} for development most probable paths with the statements in Section~\ref{sec:TrueMPPs} regarding most probable paths in the usual sense. In this case, $g$ is a Riemannian metric and $\sharp$ becomes an invertible map. We can choose $\nabla$ to be the Levi-Civita connection of $g$. There are now no singular curves by Remark~\ref{re:NoSingular}. Let $(\nabla u_t)^\dagger:TM \to TM$ be defined by $\langle (\nabla u_t)^\dagger v_1, v_2 \rangle_g = \langle \nabla_{v_2} u_t, v_1 \rangle_g$. The equations \eqref{MPPu} can then be written as
\begin{equation*}   \left\{ \begin{aligned}
  \frac{D}{dt} \Sigma_t^{-1} (\dot \gamma_t -u_t(\gamma_t)) & =  - (\nabla u_t)^\dagger \Sigma_t^{-1} (\dot \gamma_t -u_t(\gamma_t))+ R(\chi_t) \dot \gamma_t,  \\
  \frac{D}{dt} \chi_t & =  \frac{1}{2} (\dot \gamma_t -u_t(\gamma_t)) \wedge \Sigma_t^{-1} (\dot \gamma_t -u_t(\gamma_t)), \qquad \chi_T = 0.
\end{aligned} \right.
\end{equation*}
Note that, in this case, the equations are generalizations of equations of dMPPs found in \cite{GrSo22}. If $\Sigma_t =\id$ and $u = u_t$ becomes time-homogeneous, the resulting formula for the dMPPs is
\begin{equation*}  
  \frac{D}{dt} \dot \gamma +  (\nabla u)^\dagger \dot \gamma_t  = \nabla_{\dot \gamma_t} u + (\nabla u)^\dagger u(\gamma_t),
\end{equation*}
If we compare these equations with the MPPs, i.e., the curves that are critical points of \eqref{OMRie}, we find that the latter are given by
\begin{equation*}  
  \frac{D}{dt} \dot \gamma +  (\nabla u)^\dagger \dot \gamma_t  = \nabla_{\dot \gamma_t} u + (\nabla u)^\dagger u(\gamma_t) + \nabla F,
\end{equation*}
where $F = \DIV(u) - \frac{1}{12} S$ consist of the divergence of $u$ and the scalar curvature, which exactly corresponds to changes in volume.
\end{remark}

\subsection{Optimal control point of view dMPPs} \label{sec:OptimalControl}
Finding development most probable paths can be formulated as an optimal control problem in the following way. We want to consider finding curves $\gamma_t$ solving \eqref{gammatFirsteq} with final point $\gamma_T =y$ and minimizing $\OM_{E_o}(\omega_t) = \frac{1}{2} \int_0^T \langle \dot \omega_t, \dot \omega_t \rangle_{g(o)} \, dt$ as an optimal control problem. For this formulation, we will use the structure of the orthogonal frame bundle of $E$. For analogous definitions of the orthonormal frame bundle of the full tangent space, see \cite[Chapter 2]{Hsu02}. For construction on a subbundle, see \cite[Section 2]{Grong2014CurvaturedimensionIO}.

Assume that the bundle $E$ has rank $J \leq d$. For the sake of simplicity, we will assume that $u_t$ has values in $E$ as well, see Remark~\ref{re:UnotinE} for when this is not the case. Let $\Ort(J) \to \Ort(E) \stackrel{\pi}{\to} M$ be the bundle of orthonormal frames of $E$. Recall that at each $y \in M$, the fiber $\Ort(E) = \pi^{-1}(y)$ consists of linear isometries $f = (f_1, \dots, f_J): \mathbb{R}^J \to E_y$, with~$\Ort(J)$ acting by compositions of the right. Let $f_t = (f_{1,t}, \dots, f_{J,t})$ be an orthonormal frame of $E$ along the curve $\gamma_t$, defined on $(-\ve, \ve)$. Assume that each vector field $f_{j,t}$ is parallel along $\gamma_t$ with respect to some given compatible connection $\nabla$. We introduce the following notation.
\begin{enumerate}[$\bullet$]
    \item If $\dot \gamma_0 = v \in E_{\gamma_0}$, we write $h_{f_0} v =\dot f_0 = h_{f_0} v$.
    \item We write $H_j(f_0) = h_{f_0} f_{j,0}$.
\end{enumerate}
The vector $h_{f_0} v$ is called \emph{the horizontal lift} of the vector $v \in T_{\pi(f_0)} M$. If~$X$ is a vector field on $M$ with values in $E$, we write $hX(f) = h_{f} X(\pi(f))$, which is \emph{the horizontal lift} of~$X$ with respect to $\hat E  = \hat E^\nabla = \spn \{ h_f v \, : \, f\in \Ort(E), v \in T_{\pi(f)} M \}$. Define a vector field
$$H_j: f \mapsto H_j(f) \qquad j=1, \dots, J \qquad \text{on $\Ort(E)$.}$$
Note that $H_1, \dots, H_J$ are globally defined vector fields that give a basis of $\hat E^\nabla$, but cannot be projected to vector fields on $M$.

Write $\dot \omega_{j,t} = \nu_{j,t}$ and let $f_t$ be a parallel curve along~$\gamma_t$, where $\gamma_t$ is the solution of \eqref{gammatFirsteq}. If we introduce notation $\Sigma_t^{1/2} f_{j,0} = \sum_{i=1}^J (\Sigma^{1/2}_t)_{ij} f_{i,0}$, then $f_t$ is the solution of
\begin{equation} \label{control} \dot f_t = \sum_{i,j=1}^J (\Sigma_t^{1/2})_{ij} \nu_{j,t} H_i(f_t) + hu_t(f_t).\end{equation}
We can consider \eqref{control} as a control system on $\Ort(E)$, and finding a developed most probable path as the result of finding an optimal control $\nu_t = (\nu_{1,t}, \dots, \nu_{J,t})$ such that $f_t$ becomes a curve from $f_0$ to $\pi^{-1}(y)$ that minimize the cost function $C(\nu) = \frac{1}{2} \int_0^T \langle \nu_t, \nu_t \rangle_g \, dt$.

This optimization problem can also be described in the following way. Define $\hat \rho_t$ as the time-dependent sub-Riemannian metric on $\hat E$ such that $\sum_{i=1}^J (\Sigma_t^{1/2})_{ij} H_i$, $j=1,2, \dots, J$ form an orthonormal basis. A developed most probable path is then a projection of a curve $f_t$ from $f_0 \in E_0$ to $\Ort(E)_y =\pi^{-1}(y)$ minimizing the energy
$$\tilde {\mathscr{E}}(f_t) = \frac{1}{2}\int_0^T \left\| \dot f_t - hu_t(f_t) \right\|_{\tilde \rho_t}^2 \, dt .$$
If $u_t =0$ and $\hat \rho_t =\hat \rho$, such curves are length minimizing curves of the sub-Riemannian structure $(\hat E, \hat \rho)$.

While $\hat E$ is projected to $E$ by $\pi_*$ bijectively on every fiber, the metric $\hat \rho_t$ is in general not the pull-back of a time-dependent sub-Riemannian metric $\rho_t$ on $E$. We give two notable exceptions to this.
\begin{enumerate}[(I)]
    \item {\it Isotropic covariance:} If $\Sigma_t = c_t \id_{E_o}$, then $\hat \rho_t = \pi^* \rho_t$ with $\rho_t = c_t^{-1} g$. This is reflected in equation \eqref{MPPu} where $\chi_t =0$ if we have isotropic covariance.
    \item {\it Flat connection:} If $\nabla$ is flat on $E$ and $M$ is simply connected, then parallel transport is path independent. Hence, we can parallel transport $\Sigma_t$ to every point giving us an endomorphism $\Sigma_t: E \to E$ which we denote by the same symbol. We then have that $\tilde \rho_t = \pi^* \rho_t$, with $\rho_t = \langle \Sigma_t^{-1} \cdot , \cdot \rangle_g$. This is reflected in equation \eqref{MPPu} in that $K=0$ when $\nabla|E$ is flat, and hence $\chi_t$ play no role in this case. If $M$ is not simply connected, we still have the same identities locally.
\end{enumerate}
If $\hat \rho_t = \pi^* \rho_t$, then development most probable paths are curves from $o$ to $y$ minimizing $\mathscr{E}(\gamma_t) = \frac{1}{2} \int_0^T \| \dot \gamma_t - u_t(\gamma_t)\|_{\rho_t}^2 \, dt$, that is, they become usual sub-Riemannian geodesics on on $M$ with respect to $(E, \rho_t)$.

\begin{remark}
Considering the actual question of controllability, that is, if we can reach any point $y$ in $M$ from $o$ with a developed curve, we know that if $u_t = u$ is time-homogeneous, then we can find a curve $\dot \gamma_t = \ptr_t \int_0^t \Sigma_s^{1/2} \dot \omega_s \, ds + u(\gamma_t)$ connecting any pair of points if $E$ is bracket-generating, see \cite[Proposition~26]{boscain2019introduction}. We call $E$ bracket-generating if the sections of $E$ and their iterated brackets span the entire tangent bundle $TM$.
\end{remark}

\begin{remark} \label{re:NonMarkov}
Assume that $\Sigma_t = \Sigma$ and and $u_t = u$ are independent of time. If $f_0$ is a choice of orthonormal frame $E_o$, we write $\Sigma^{1/2}f_{j,0} = \sum_{i=1}^J (\Sigma^{1/2})_{ij} f_{i,0}$. Then solving $x_t = \ptr_t \Sigma^{1/2} \circ dB_t + u(x_t) dt$ can be considered as first solving
$$d\hat x_t = \sum_{i,j=1}^J (\Sigma^{1/2})_{ij} H_i(\hat x_t) \circ d\hat B_t^j + hu(\hat x_t), \qquad \hat x_0 = f_0,$$
with $x_t = \pi(\hat x_t)$. While $\hat x_t$ is always Markov, the same does not hold for $x_t$, again because in general $\Sigma$ is not rotationally invariant and because parallel transport depends on paths. The cases (I) and (II) above respectively describes the exceptions, which also give that $x_t$ har the Markov property.
\end{remark}

\begin{remark} \label{re:UnotinE}
If $u_t$ does not take values in $E$, we can still provide a similar formulation, by considering the submanifold $P \subseteq \GL(TM)$ of the general frame bundle consisting of frames $f_1, \dots, f_d$ where $f_1, \dots, f_J$ is an orthonormal frame of $E$.
\end{remark}

\subsection{Change of connection}
Since we have used the notion of development to define our developed most-probable paths, we need to show that this definition does not depend on the connection if the results are the same. To be specific several different choices of affine connection that will give us the same solution of
$$dx_t = \ptr_t \Sigma_t^{1/2} \circ dB_t + u_t(\gamma_t) dt, \qquad x_0 = o \in M, \qquad \text{$B_t$ a Brownian motion in $E_o$,}$$
so we should verify that dMPPs are indeed invariant under such choices. 
\begin{proposition} \label{prop:ConVert}
Let $\nabla$ be a given connection that is compatible with a sub-Riemannian structure $E,g$. Let $\tilde \nabla$ be another connection satisfying
\begin{equation} \label{TwoConnections} \tilde \nabla_X Y := \nabla_X Y + \mu_X Y \qquad \mu_X|E =0, \end{equation}
where $X \mapsto \mu_X$ is a $TM$-endomorphism-valued one-form. Then the solutions of the equations \eqref{singular} for singular curves and \eqref{MPPu} for non-singular dMPPs coincides for these connections.
\end{proposition}

\begin{proof}
By definition, $K$ in \eqref{defK} only depends on the restriction of the connection to $E$, so it is sufficient to consider equation \eqref{singular} for singular curve.  Here we observe that
\begin{align*}
    & \frac{\tilde D}{dt} \lambda_t = \frac{D}{dt}\lambda_t - \lambda_t \mu_{\dot \gamma_t}   
\end{align*}
and since $\mu_{X} \dot \gamma_t = \mu_{X} u_t(\gamma_t)$ for any vector field $X$, if we write \eqref{singular} relative to $\tilde \nabla$, then
\begin{align*}
      - \lambda_t (\tilde \nabla u_t) - \lambda_t \tilde T(\dot \gamma_t, \cdot) 
   & = -\lambda_t(\nabla u_t) -\lambda_t(\mu_{\cdot} u_t) - \lambda_t(T(\dot \gamma_t, \cdot )) - \lambda_t \mu_{\gamma_t} +\lambda_t(\mu_{\cdot} u_t) \\
    & = -\lambda_t(\nabla u_t)  - \lambda_t(T(\dot \gamma_t, \cdot )) - \lambda_t \mu_{\gamma_t} 
\end{align*}
which clearly gives the same solutions as those of \eqref{singular} written with respect to $\nabla$.
\end{proof}

Furthermore, the concept of singular curves is not related to the choice of affine connection at all.
\begin{lemma}
Let $\nabla$ and $\tilde \nabla$ be two connections that are compatible with the sub-Riemannian structure $(E,g)$ and let $u_t$ be a vector field. Let $\Dev_u$ and $\widetilde{\Dev}_u$ be the corresponding affine development. Then these maps have the same singular curves. Hence, the definition of singular curves, only depends on the triple $(E,g,u_t)$.
\end{lemma}

\begin{proof}
Let $\gamma_t$ be a curve that is singular with respect to $\nabla$, and let $\lambda_t$ be a one-form along $\gamma_t$ satisfying~\eqref{singular}. If we write $\tilde \nabla_X Y = \nabla_X Y + \mu_X Y$, then since $\sharp \lambda_t = 0$, we have $\lambda_t(\mu_X \dot \gamma )= \lambda_t(\mu_X u_t)$. Here we have used that $\mu_X(E) \subseteq E$ for any vector field $X$, since we are assuming that both connections are compatible with the sub-Riemannian structure. It follows that
\begin{align*}
    & \frac{\tilde D}{dt} \lambda_t + \lambda_t \tilde \nabla u_t + \lambda_t \tilde T(\dot \gamma, \cdot ) \\
    & = \frac{D}{dt} \lambda_t - \lambda_t\mu_{\dot \gamma}  + \lambda_t \nabla u_t + \lambda_t \mu_{\cdot} u_t + \lambda_t T(\dot \gamma, \cdot) + \lambda_t \mu_{\dot \gamma_t} - \lambda_t \mu_{\cdot} \dot \gamma_t =0.
\end{align*}
Hence $\gamma_t$ is singular with respect to $\tilde \nabla$ as well.
\end{proof}

\section{Examples}
\label{sec:examples}
In this section, we consider several different examples, illustrating the development most probable paths equations (dMPPs) in Theorem~\ref{th:MPPu} in specific cases.
\subsection{Riemannian and sub-Riemannian anisotropic Brownian motion} \label{sec:RSR}
Let $(M,g)$ be a Riemannian manifold. For a given $o \in M$, let $B_t$ be the Brownian motion in $T_o M$. Write $\nabla$ for the Levi-Civita connection $g$. For a symmetric endomorphism $\Sigma: T_o M \to T_o M$, define $x_t$ as the solution of
\begin{equation} \label{AnisBM} dx_t = \ptr_t \Sigma^{1/2} \circ dB_t.\end{equation}
We consider the developed most probable paths of this process. Recall that we are using the symbol $\Sigma: T_{\gamma_t}M \to T_{\gamma_t}M$ for its parallel transport along $\gamma_t$ as well. By Remark~\ref{re:NoSingular}, there are no singular curves, and any dMPP is the solution of 
\begin{equation*} 
\frac{D}{dt}\dot \gamma_t   = \Sigma R(\chi_t)\dot \gamma_t, \qquad 
\frac{D}{dt} \chi_t  = \frac{1}{2} \dot \gamma_t \wedge \Sigma^{-1} \dot \gamma_t, \qquad \chi_T =0.
\end{equation*}
These are exactly the equations found earlier in \cite{grongMostProbablePaths2022}. We can also consider a sub-Riemannian analogue. We look at \eqref{AnisBM} for a connection $\nabla$ that is compatible with a sub-Riemannian structure $(E,g)$ and with $u = 0$. Equations for the most probable path for non-singular curves are given by
\begin{equation*}   \left\{ \begin{aligned}
  \sharp^g \lambda_t & = \Sigma^{-1} \dot \gamma_t \\
  \frac{D}{dt} \lambda_t & = - \lambda_t(T(\dot \gamma_t, \cdot)) + K(\chi_t, \gamma_t),  \\
  \frac{D}{dt} \chi_t & = \frac{1}{2} \dot \gamma_t \wedge \Sigma^{-1} \dot \gamma_t, \qquad \chi_T = 0.
\end{aligned} \right.
\end{equation*}
If either condition (I) or (II) in Section~\ref{sec:OptimalControl} is satisfied so that $\chi_t =0$, then the above equations reduce to the equations for non-singular geodesics in sub-Riemannian manifolds \cite{godoy2017riemannian}.

\subsection{Lie groups with a right invariant structure} \label{sec:Lie}
Let $G$ be a Lie group with Lie algebra $(\mathfrak{g}, [\cdot, \cdot ]_{\Lie})$. We will use the same symbol for elements in $\mathfrak{g}$ and their corresponding right invariant vector field on $G$. Let $E_{1_G} = \mathfrak{e} \subseteq \mathfrak{g}$ be a subspace with an inner product $\langle \cdot , \cdot \rangle$ and define a sub-Riemannian structure $(E,g)$ by right translation. Let $\nabla$ be the right invariant connection on $G$. Recall that this connection is flat and with torsion such that $T(A,B) = [A,B]_{\Lie} =\ad(A)B = -\ad(B)A$ for any right invariant vector fields $A$ and $B$. Let $u_t$ be a vector field defined by right translation of a curve $a_t$ in $\mathfrak{g}$. Finally, consider a symmetric transformation $\Sigma_t$ of $\mathfrak{e}$. Let $B_t$ be a Brownian motion in $\mathfrak{e}$, and consider $x_t$ as the solution of $x_0 = 1_G$,
$$dx_t = \ptr_t \Sigma_t^{1/2} \circ dB_t + u_t(x_t) \, dt = \circ dZ_t \cdot x_t,$$
with $Z_t = \int_0^t \Sigma_s^{1/2} \circ dB_s + \int_0^t a_s ds$. Since $\nabla$ is a flat connection, the development most probable paths are solutions of
\begin{equation} \label{LieMPP} \left\{ \begin{aligned}
  \dot \gamma_t & = z_t \cdot \gamma_t,\\
 \sharp^g \alpha_t &= \Sigma_t^{-1} (z_t -a_t), \\
  \dot \alpha_t & = - (\ad(z_t))^*\alpha_t,
\end{aligned} \right.
\end{equation}
where $z_t$ and $\alpha_t$ are respectively curves $\mathfrak{g}$ and $\mathfrak{g}^*$ and $\ad(z_t)^*$ denotes the pull-back. Singular curves will be the curves where there is a non-trivial solution to
\begin{equation} \label{LieSingular} \left\{ \begin{aligned}
  \dot \gamma_t & = z_t \cdot \gamma_t,\\
  z_t-a_t & \in \mathfrak{e}, \\
 \sharp^g \alpha_t &= 0, \\
  \dot \alpha_t & = - (\ad(z_t))^*\alpha_t.
\end{aligned} \right.
\end{equation}

\begin{remark}
If we replace right with left invariance in the above discussion, we get similar formulas but with $\ad$ replaced with $-\ad$.
\end{remark}

\subsection{Metrics invariant by a subgroup}
We look at the following particular case of Section~\ref{sec:Lie}. Assume that $\mathfrak{g} = \mathfrak{e} \oplus \mathfrak{k}$ for some subalgebra $\mathfrak{k}$. Assume furthermore that $\mathfrak{g}$ has an inner product that is invariant under the adjoint action of $\mathfrak{k}$, with $\mathfrak{e} = \mathfrak{k}^\perp$. It follows that $[\mathfrak{k}, \mathfrak{e}] \subseteq \mathfrak{e}$. We will also assume that $[\mathfrak{e}, \mathfrak{e}] \subseteq \mathfrak{k}$. We define a Riemannian metric $\bar{g}$ by left translation of the inner product on $\mathfrak{g}$. Write $E_x = x \cdot \mathfrak{e}$ and $g = \bar{g}|_E$. If we define $w_t = \sharp^{\bar{g}} \alpha_t = w_{\mathfrak{e},t} + w_{\mathfrak{k},t}$, then
\begin{equation*} \left\{ \begin{aligned}
  \dot \gamma_t & = \gamma_t \cdot (\Sigma_t w_{\mathfrak{e},t} + a_{t})\\
  \dot w_{\mathfrak{e},t} & = (\ad(\Sigma_t w_{\mathfrak{e},t} + a_{\mathfrak{e},t}))^\dagger w_{\mathfrak{k},t} + (\ad(a_{\mathfrak{k},t}))^\dagger w_{\mathfrak{e},t}, \\
  \dot w_{\mathfrak{k},t} & = (\ad(\Sigma_t w_{\mathfrak{e},t} + a_{\mathfrak{e},t}))^\dagger w_{\mathfrak{e},t} + (\ad(a_{\mathfrak{k},t}))^\dagger w_{\mathfrak{k},t},,
\end{aligned} \right.
\end{equation*}
where $\ad(z)^\dagger$ is the adjoint with respect to $\langle \cdot , \cdot \rangle$.

\subsection{Locally homogeneous spaces} For our final, we assume that $(M,E,g)$ is a sub-Riemannian manifold with a compatible connection $\nabla$ such that $u_t$, as well as its curvature $R$ and and torsion $T$ are all parallel. We can solve then solve equations \eqref{MPPu} at one point $o \in M$.
\begin{equation*}   \left\{ \begin{aligned}
  \dot \gamma_t & = \ptr_t (\Sigma_t \sharp^g \alpha_t + a_t), \qquad a_t = u_t(o) \\
  \dot \alpha_t & = - \alpha_t(T(\Sigma_t \sharp^g \alpha_t + a_t, \cdot)) + K(\check{\chi}_t, \Sigma_t \sharp^g \alpha_t + a_t),  \\
  \dot {\check{\chi}}_t & = \frac{1}{2} \Sigma_t \sharp^g \alpha_t \wedge \sharp^g \alpha_t, \qquad \alpha_T = 0,
\end{aligned} \right.
\end{equation*}
where $\alpha_t$ and $\check{\chi}_t$ are curves in respectively $T_o^* M$ and $\wedge^2 E_o$. 

In particular, if $u =0$, and both $R$ and $T$ are parallel, then $M$ can locally be written as a homogeneous space, i.e., $(M, \nabla)$ is locally affinely equivalent to the canonical connection on a space $G/K$.

\section{Submersions}
\label{sec:submersions}

We here study development most probable paths between points in the image of submersions $\pi:M \to N$. The motivation is to describe the relations between dMPPs in $M$ and $N$, similar to the relations between geodesics on Riemannian submersions described by O'Neill in \cite{MR216432}. With this description, we prepare a formalism to use in Section~\ref{sec:diffeos} for a submersion from an infinite-dimensional manifold to a finite dimensional manifold.

\subsection{General submersions}
Let $M$ be a manifold with a sub-Riemannian structure $(E,g)$. Let $\nabla$ be a connection compatible with the sub-Riemannian structure. Let $\Sigma_t: E_o \to E_o$ be a symmetric transformation of the space $E_o$ at $o \in M$.
Relative to a time-dependent vector field, we define
$$dx_t = \ptr_t \Sigma_t^{1/2} \circ dB_t +  u_t(x_t) dt, \qquad x_0 = o,$$
with $B_t$ being a Brownian motion in $E_o$. As before, we are studying curves $\gamma_t$ from $o$ to $y = \gamma_T= \Pi(\omega)$ solving $\dot \gamma_t =\ptr_t \Sigma_t^{1/2} \dot \omega_t+ u_t(\gamma_t)$ such that $\omega_t$ is most probable with respect to $B_t$ in $\Pi^{-1}(y)$.
Let $\pi: M \to N$ be a surjective submersion and write $\calV = \ker \pi_*$ for the vertical space. We first note the following relation.
\begin{lemma}
If $\check{\gamma}_t = \pi(\gamma_t)$ is a curve from $\check{o} = \pi(o)$ to $\check{y} = \pi(y) = (\pi \circ \Pi)(\omega)$, and if $\omega$ is a regular point of $\pi \circ \Pi$ and a most probable path in $(\pi \circ \Pi)^{-1}(\check{y})$, then $\gamma_t$ is a solution \eqref{MPPu} with $\lambda_T$ vanishing on $\calV_{y}$
\end{lemma}
\begin{proof}
This follows from considering the proof of Theorem~\ref{th:MPPu}, but with variations $\omega^s$ of $\omega$ such that if $\dot{\gamma}_t^s = u_t(\gamma_t^s) + \ptr_t^s \Sigma_t^{1/2} \dot \omega_t^s$, then $\frac{d}{ds} \gamma_t^s|_{s=0} \in \calV_y$.
\end{proof}

We consider the following special case. For the surjective submersion $\pi: M \to N$, if $\check y \in N$, write $M_{\check y} = \pi^{-1}(\check y)$. Let $(\check E, \check g)$ be a sub-Riemannian structure on $N$ such that $\pi_*|_E: E \to \check{E}$ is surjective with
$$\langle v, w \rangle_g = \langle \pi_* v, \pi_*w \rangle_{\check{g}}, \qquad v,w \in \calH := (\calV \cap E)^\perp,$$
where the orthogonal complement is taken within $E$. In other words, we can write $E = \calH \oplus_{\perp} \calV \cap E$ with $\pi_*$ restricted to $\calH_y$ being a linear isometry onto $\check E_{\pi(y)}$. For a vector $w \in \check{E}_{\check y}$ and $y \in M_{\check y}$, we define $h_y w$ as the unique element in $\calH_y$ that projects to $w$. For a vector field $X$ with values in $\check E$, we define $hX(y) = h_y X(\pi(y))$.

Let $o \in M$ be a point with $\pi(o) = \check{o}$. We let 
$\Sigma_t$ be an invertible, symmetric positive map of $E$. We assume that $\Sigma_t$ respect the splitting $E_o = \calH_o \oplus (E_o \cap \calV_o)$ in the sense that $\langle \Sigma_t v, w \rangle_g = \langle  v, \Sigma_t w \rangle_g =0$ for any $v \in \calH_o$ and $w \in E_o \cap \calV_o$. It then follows from our assumptions on $\Sigma_t$ that for any $v \in \calH_o$,
$$\Sigma_t v = h_o \check \Sigma_t \pi_* v,$$
for some symmetric map $\check \Sigma_t: \check E_{\check o} \to \check E_{\check o}$. Finally, we will assume that the following holds.
\begin{enumerate}[$\bullet$]
    \item The vector field $u_t$ takes values in $E$ and is $\pi$-related to a vector field $\check{u}_t$ on $N$, i.e., we have
    $$\pi_* u_t = \check{u}_t \circ \pi.$$
    \item Identifying $\calH$ with $\pi^*\check{E}$, we have that $\calH$ is parallel with respect to $\nabla$ and furthermore, the restriction of $\nabla$ to $\calH$ is the pull-back connection of $\check{\nabla}|\check{E}$ for some connection $\check{\nabla}$ on $N$ which is compatible with $(\check{E},\check g)$. In other words, for any vector field $Y$ on $N$, we have
    \begin{equation} \label{PBC} \nabla_{v} hY = h \check{\nabla}_{\pi_* v} Y, \qquad v\in TM.\end{equation}
\end{enumerate}
If we define $\check{x}_t = \pi(x_t)$, we now obtain an analogous SDE on the base space
\begin{equation} \label{checkxt} d\check{x}_t = \cptr_t \check\Sigma_r^{1/2} \circ d\check B_t + \check{u}_t(\check{x}_t) \, dt.\end{equation}
where $\pi_* B_t = \check{B}_t$ is a Brownian motion in $\check E_{\check o}$. This leads to the following result.

\begin{proposition} \label{prop:Lifted1}
If $\check{\gamma}_t$ is a developed most probable path of $\check{x}_t$ starting at $\check{o}$ and not singular, then it is the projection of a solution $\gamma_t$ of \eqref{MPPu} starting at $o$ with $\lambda_0$ vanishing on~$\calV_o$.
\end{proposition}
In other words, the dMPPs of $\check{x}_t$ are projections of the projections of the dMPPs of $x_t$ where the initial condition of the covector vanishes on $\calV_o$; paralelling O'Neills result in \cite{MR216432} for Riemannian submersion and geodesics.

\begin{proof}
Since $\calH$ is parallel with respect to $\nabla$, this implies that $E \cap \calV$ is parallel from compatibility of the connection with the metric. Using Proposition~\ref{prop:ConVert}, we are free to modify the connection $\nabla$ outside of $E\subseteq TM$. Choose a complement~$\calP$ to~$\calV$ in $TM$ such that $\calH \subseteq \calP$. With respect to~$\calP$, we have well-defined horizontal lifts for all elements in $TN$. Without loss of generality, we may assume that~$\nabla$ on~$\calP$ is the pull-back connection of $\check \nabla$ on~$TN$, in the sense that is satisfies \eqref{PBC} for horizontal lifts of all vector fields. We may also assume that $\calV$ is parallel. 

We note the following about the torsion of $\nabla$. Let $X$ and $Y$ be vector fields on~$N$, while $V$ and $W$ are sections of $\calV$. Let $\calR$ be the curvature of $\calP$, that is the tensor
\begin{equation} 
\label{calR} \calR(hX+V,Y+W) = \pr_{\calV} [hX, hY],
\end{equation}
where $\pr_{\calV}$ is the projection to $\calV$ with kernel $\calP$. Since $hX$ and $hY$ are $\pi$-related to vector fields $X$ and $Y$ on $N$, we have
$$[hX,hY] = h[X,Y] + \calR(hX, hY).$$
Furthermore, we note that $[hX,V]$ will be in $\calV$ for any vector field $V$ with $\calV$, since $V$ is $\pi$-related to $0$. Because $\nabla$ is the pull-back connection on $\calP$, $\nabla_{hX} hY = h\check \nabla_X Y$ and $\nabla_V hX =0$. From these identities, it follows that for vector fields $X$, $Y$ in $\calV$ and vector fields $V$ and $W$ with values in $\calV$,
$$T(hX, hY) = h \check T(X,Y) + \calR(hX,hY), \qquad T(hX, V)|_x, T(V,W)|_x  \in \calV .$$
In particular, if $\alpha \in T^*M$ vanishes on $\calV$, then $\alpha(T(v,w)) = \alpha(h\check T(\pi_* v, \pi_* w))$.

Write $u_t = h\check u_t + u^\calV_t$. Let $(\check{\gamma}_t, \check \lambda_t, \check \chi_t)$ be a solution to \eqref{MPPu} on $N$. Define $(\gamma_t, \lambda_t, \chi_t)$ by
\begin{align*}
    \dot \gamma_t & = h_{\gamma_t} (\dot{\check{\gamma}}_t -\check u_t(\check \gamma_t)) + u_t(\gamma_t), \\
    \lambda_t & = \pi_{\gamma_t}^* \check \lambda_t, \\
    \chi_t & = h_{\gamma_t} \check \chi_t.
\end{align*}
By this definition, $\sharp^g \lambda_t =  \Sigma_t^{-1} (\dot \gamma - u_t(\gamma_t))$, $\frac{D}{dt} \chi_t = h_{\gamma_t} \frac{\check D}{dt} \chi_t = \Sigma_t \sharp^g \lambda_t \wedge \sharp^g \lambda_t$, and furthermore,
\begin{align*}
    & \frac{D}{dt} \lambda_t + \lambda_t(\nabla u_t) + \lambda_t(T(\dot \gamma_t, \cdot)) + K(\chi_t, \dot \gamma_t) \\
    & \qquad= \pi^*_{\gamma_t} \frac{\check D}{dt} \check \lambda_t + \check \lambda_t(\check \nabla_{\pi_t} \check u_t) + \check \lambda_t(\check T(\dot {\check \gamma}_t,\pi_*  \cdot) + \pi^*_{\gamma_t} \check K(\check \chi_t, \dot {\check \gamma}_t) =0.
\end{align*}
This completes the proof of the result.
\end{proof}

With a similar approach, we have the following.
\begin{proposition}\label{prop:Lifted2}
Under the same assumptions as in Proposition~\ref{prop:Lifted1}, the curve $\check{\gamma}$ is singular if and only if it is the projection of a solution $\gamma$ of the equations \eqref{singular} for singular curves on $M$, which in addition satisfies $\lambda_0$ vanishing on $E+\calV$.
\end{proposition}

\begin{remark} \label{re:CanTakeH}
For Proposition~\ref{prop:Lifted1} and Proposition~\ref{prop:Lifted2}, if we are only interested in finding lifted formulas for singular curves and most probable paths on $N$, we can without loss of generality assume that $E = \calH$, as $E\cap \calV$ will play no role in the equations. We can also simplify to $u_t = h \check u_t$
\end{remark}

\subsection{Examples: Homogeneous spaces} \label{sec:ExHom}
As an example of submersion we consider a homogeneous space. Let $G$ be a Lie group with Lie algebra $\mathfrak{g}$. Let $\mathfrak{k}$ be a subalgebra corresponding to a closed subgroup $K$, and define $N = G/K$. We will assume that we have an inner product $\langle \cdot , \cdot \rangle$ on $\mathfrak{g}$. Write $\mathfrak{p} = \mathfrak{k}^\perp$ for the orthogonal complement of $\mathfrak{k}$. We observe that if $\pi: G \to N = G/K$ is the canonical projection then
$$\calV_x = \ker \pi_{*,x} = x \cdot \mathfrak{k}.$$
If $A$ is a non-zero left-invariant vector field with $A|_{1_G} \in \mathfrak{p}$, then $\pi_* A$ is non-vanishing, but only a vector field on $N$ if $A|_{1_G}$ is invariant under the action of $K$. On the other hand, if $A$ is a right invariant vector field, then $\pi_* A$ does define a vector field on $N$, but this vector field can vanish at some points.

We consider the following processes $\check{x}_t$ on homogeneous spaces that are projections of process $x_t$ on Lie groups.

\begin{example} \label{submersion}
In the first example we consider left-invariant structures on $G$, so we start by assuming that the inner product $\langle \cdot , \cdot \rangle$ is on $\mathfrak{g}$ is invariant under the adjoint action of $K$.
Define $E= \calP$ by left translation of $\mathfrak{p}$ and with a metric $g$ also obtained by left translation. Then for every $y \in G$, $\pi_{*,y}: \calP_y \to T_{\pi(y)}N$ is an invertible map which induces a Riemannian metric $\check{g}$ on $TN$ by the assumption of the invariance of the inner product under the action of $K$. Let $h_y v \in \calP_y$ denote the horizontal lift of a vector $v \in T_{\check y} N$, $y \in \pi^{-1}(\check y)$. Let $\nabla$ be the left-invariant connection on $G$, i.e., the connection uniquely defined by all left-invariant vector fields being parallel. We can define a connection $\check{\nabla}$ on $N$ by
$$\nabla_{hX} hY = h \check{\nabla}_X Y.$$
Here we have used that $\calP$ is left-invariant, and so if $hY$ is a section of $\calP$, then $\nabla_{hX} hY$ is a section of $\calP$ also.
Consider a vector field $\check{u}_t$ on $N$, a symmetric positive definite map $\check{\Sigma}_t:T_o N \to T_o N$ and a Brownian motion $\check B_t$ in $T_oN$. We define a process $\check x_t$ as a solution of
$$d\check x_t = \cptr_t \check{\Sigma}_t^{1/2} \circ d\check{B}_t + \check{u}_t(\check x_t) \, dt, \qquad \check{x}_o = \check{o}.$$
Without loss of generality, we may assume that $\pi(1_G) = \check{o}$. We can then consider $\pi(x_t) = \check x_t$, with
$$dx_t = x_t \cdot  \Sigma_t^{1/2} \circ dB_t + h\check{u}_t(x_t) \, dt, \qquad x_0 = 1_G,$$
where $\Sigma_t$ is a symmetric map of $\mathfrak{p}$ and $B_t$ is a Brownian motion in $\mathfrak{p}$. Write $\bar{u}_t(y) = y^{-1} \cdot h\check{u}_t(y)$, which satisfies $\bar{u}_t(y \cdot k) = \Ad(k^{-1}) \bar{u}_t(y)$. Then the development most probable paths in $N$ are projections of
\begin{equation} \label{HomoGenMPP} \left\{ \begin{aligned}
  \dot \gamma_t & =\gamma_t \cdot z_t, \qquad \gamma_0 =1\\
  z_t &= \Sigma_t \sharp^g \alpha_t + \bar{u}_t(\gamma_t), \\
  \dot \alpha_t & = (\ad(z_t))^*\alpha_t, \qquad \alpha_0|_\mathfrak{k} =0.
\end{aligned} \right.
\end{equation}
\end{example}

\begin{example}
Consider now the case when $\tilde \nabla$ is a right invariant connection and define $(E,g)$ by right translation of a subspace $\mathfrak{e} \subseteq \mathfrak{p}$ and its inner product. We assume that there is a subbundle $\check E$ on $N$ such that $\pi_* E = \check E$ is well-defined and bijective on every fiber. This happens if $\Ad(x) \mathfrak{e}$ is transverse for $\mathfrak{k}$ for any $x$ in $G$. If this assumption holds and if $A$ is a right invariant vector fields with values in~$E$, then $\pi_* A$ is a non-vanishing vector field on $N$.

Let $a_t$ be a curve in $\mathfrak{g}$. Define $u_t$ as the vector field on $G$ obtained by right translation of $a_t$ and define $\check u_t = \pi_* u_t$ which is well defined at every point. Consider the stochastic process
$$dx_t = {\tptr}_t \Sigma_t^{1/2} \circ dB_t + u_t(x_t) \, dt = (\Sigma_t^{1/2} \circ dB_t + a_t(x_t) dt) \cdot x_t, \qquad x_0 = 1_G .$$
Let $\nabla$ be a connection on $N$ such that $\pi_* A$ is parallel for any right-invariant vector field. Then $\check{x}_t = \pi(x_t)$ is the solution of
$$d\check x_t = \ptr_t \check \Sigma_t^{1/2} \circ d\check B_t + \check u_t(x_t) \, dt,$$
where $\check B_t = \pi_{*,o} B_t$ is a Brownian motion in $\check E_{\check o}$ and with $\check \Sigma_t \pi_{*,o} = \pi_{*,o} \Sigma_t$. Singular curves and development most probable paths related to $\check x_t$ are then given by solutions of respectively \eqref{LieSingular} and \eqref{LieMPP} with $\alpha$ vanishing on $\mathfrak{k}$.
\end{example}

\section{Sobolev diffeomorphism and landmarks}
\label{sec:diffeos}
We here cover development most probable paths on the diffeomorphism group, as well as in the space of landmark configurations, where we have an action by diffeomorphisms. The setting includes Kunita stochastic differential equations \cite{kunitastochasticflowsstochastic1997} with finite-dimensional driving noise. Stochastics for such systems are studied and used in fields including fluid dynamics and shape analysis, e.g., for modelling stochastic deformations of shape occurring in evolutionary biology. One example of a particular stochastic system of this form are the perturbed Hamiltonian equations of \cite{ACC14}. Other examples of related shape stochastics from action of stochastic diffeomorphisms include \cite{arnaudonGeometricFrameworkStochastic2019}.

\subsection{Sobolev diffeomorphisms} \label{sec:SobDiffeo}
We will now introduce metrics on Sobolev diffeomorphisms, see e.g., \cite{bruveris2017completeness,IKT16} for more details. 
Let $(M,g)$ be a $d$-dimensional Riemannian manifold. Let $\mathfrak{g}$ be the infinite dimensional Lie algebra consisting of $H^s$-vector fields on $M$ for $s > \frac{d}{2} +1$. For $r \leq s$, define an inner product on on $\mathfrak{g}$ by
$$\langle u, v \rangle_{H^r} = \int_M \langle u, Lv \rangle_g \, d\mu,$$
where $L$ is a positive, invertible, elliptic differential operator of order $2r$.

Let $\sigma_1, \dots, \sigma_J$ be a collection of $J$ vector fields and let $a_t$ be a time-dependent vector field. Define $\mathfrak{e} = \spn \{ \sigma_1, \dots, \sigma_J\}$, and assume for the sake of simplicity that these vector fields are all linearly independent. We emphasize that this linear independence holds for the collection as vector fields, but not necessarily pointwise. Let $B_t$ and $\hat B_t$ be Brownian motions in respectively $\mathfrak{e}$ and $\mathbb{R}^J$, and let $\Sigma: \mathfrak{e} \to \mathfrak{e}$ be the symmetric operator with
\begin{equation}
    \label{SigmaDiff} \Sigma v= \sum_{j=1}^J \langle \sigma_j , v \rangle_{H^r} \sigma_j.
\end{equation} 
We define
\begin{equation} \label{VectorFieldZt} dZ_t = a_t dt + \sum_{j=1}^J \sigma_j \circ d\hat B_t^j  = a_t dt +  \Sigma^{1/2} \circ dB_t.\end{equation}
Let $G= \mathcal{D}^s(M)$ be the group of all $C^1$-diffeomorphisms with $H^s$-Sobolev regularity. Let $\tilde \nabla$ be a right invariant connection on $G$, and define development with respect to this connection. If $(E,g)$ is the right translation of $\mathfrak{e}$ and its inner product, then~$\tilde \nabla$ is compatible with the rank $d$ sub-Riemannian structure $(E,g)$ on the infinite dimensional group $G$.

Write $x_t$ for the development of $Z_t$, that is, the solution of
\begin{equation} \label{DevOfZt}
dx_t = \circ dZ_t(x_t).
\end{equation}
Then we still have that singular curves are given by \eqref{LieSingular} and the most probable paths that are not singular are given by \eqref{LieMPP}. However, it is more complicated to verify which curves are non-singular when our ambient space is infinite dimensional. We therefore want to look at projections of this process on a finite dimensional space of landmarks in the next section.

\begin{remark}
If $a_t$, $\sigma_1$, $\dots$, $\sigma_J$, are as described above, we can consider the stochastic process on $M$
\begin{equation} \label{SDExt} dx_t = a_t(x_t) dt + \sum_{j=1}^J \sigma_j(x_t)\circ d\hat B_t^j, \qquad x_0 = \id_M,\end{equation}
which can be considered as the development of $Z_t$ in \eqref{VectorFieldZt}.
If the vector fields $\sigma_1, \dots, \sigma_J$ span $TM$, we can consider \emph{most probable transformation}, which was done in \cite[Theorem~3.1]{GrSo22}, where we are considering a curve $t \mapsto \phi_t$ in the diffeomorphism group such that each path $t \mapsto \phi_t(x)$ is most probable. We remark that the setting of \cite{GrSo22} is quite different in two ways; firstly, it does not consider that the curve $t \mapsto \phi_t$ itself is most probable, but rather its flow lines, and, secondly, it induces a Riemannian structure on $M$ from $\sigma_1, \dots, \sigma_J$ and writes the Onsager-Machlup function for the flow lines in terms of this metric.
\label{rem:GrSo22}
\end{remark}

\subsection{Applications to landmarks}
Recall that $(M,g)$ is a Riemannian $d$-dimensional manifold. For a given integer $n \geq 1$, let us define the landmark manifold
$$\calM= \{ \mathbf{x} = (x_1, \dots, x_n) \in M^{\times n} \, : \, x_i \neq x_j\},$$
which makes $\calM$ a manifold of dimension $nd$. We have an action of $G = \mathcal{D}^s$ on $\calM$ by $\phi \cdot (x_1, \dots, x_n) = (\phi(x_1) , \dots, \phi(x_n))$. Let $\mathbf{o} = (o_1, \dots, o_n) \in \calM$ be a collection of reference points, and  define $\bpi: G \to \calM$ by
$$\bpi(\phi) = \phi \cdot \mathbf{o}.$$
Define a Riemannian metric $\mathbf{g}$ on $\calM$ as the product metric induced by the Riemannian metric $g$ on $M$.

We introduce the following notation. For the time-dependent vector field $a_t$, we let $\hat a_t$ be the corresponding right invariant vector field on $G$, and define $\mathbf{a}_t$ as the vector field on $\calM$ such that
$$\mathbf{a}_t(\mathbf{x}) = (a_t(x_1), \dots, a_t(x_n)).$$
We then note that
$$\bpi_{*} \hat a_t(\phi) = \mathbf{a}_t(\phi \cdot \mathbf{o}) = \mathbf{a}_t(\bpi(\phi)).$$
In other words, $\hat a_t$ and $\mathbf{a}_t$ are $\bpi$-related. This holds for all vector fields on $M$ with similar notation.

Recall from Section~\ref{sec:SobDiffeo} that $\mathfrak{e} = \spn\{ \sigma_1, \dots, \sigma_J\}$ is the $J$-dimensional span of vector fields. Also recall that $E_{\phi} = \spn \{ \hat \sigma_1(\phi), \dots, \hat \sigma_J(\phi)\}$ was the $J$-dimensional subbundle on the infinite dimensional group $G = \mathcal{D}^s(M)$ by right translation, with $\hat \sigma_j$ being the right translation of $\sigma_j$. Introduce the following assumption.
\begin{equation} \tag{A1} \label{A1}
\begin{array}{c}
\text{The number of points that any vector field} \\
\text{in $\mathfrak{e}$ vanish at is at  most on $n-1$.}
\end{array}
\end{equation}
From assumption \eqref{A1}, it follows that $\bpi_* E = \calE$ is a rank $J$ vector bundle on $\calM$ with
$$\calE_{\mathbf{x}} =\spn\{ \bsigma_1(\mathbf{x}), \dots, \bsigma_J(\mathbf{x})\}, \qquad \bsigma_j(x) = (\sigma(x_1), \dots, \sigma(x_n)).$$
We remark that $J$ may in general be less than $nd$, making $\calE$ a proper subbundle of $T\calM$. We define sub-Riemannian metrics $\langle \cdot , \cdot \rangle_\rho$ and $\langle \cdot , \cdot \rangle_{\brho}$ on respectively $E$ and $\calE$ by
$$\langle \hat \sigma_i, \hat \sigma_j \rangle_\rho(\varphi) = \langle \bsigma_i, \bsigma_j \rangle_{\brho} (\bpi(\varphi)) = \langle \sigma_i, \sigma_j \rangle_{H^r}, \qquad
\text{for any $\varphi \in G$}.$$
If $x_t$ is the solution equations \eqref{VectorFieldZt} and \eqref{DevOfZt} on $G$, and $\bpi(x_t) = \mathbf{x}_t$, then $\mathbf{x}_t$ is the solutions
$$d\mathbf{x}_t = \sum_{j=1}^n \bsigma_j(\mathbf{x}_t) \circ d\hat B_t^j + \mathbf{a}_t(\mathbf{x}_t) \, dt,$$
with $\hat B_t = (\hat B_t^1, \dots, \hat B_t^J)$ being a standard Brownian motion.

Let $\nabla$ be a connection on $\calM$ such that $\nabla \bsigma_j =0$. Introduce coefficients
$$\Sigma_{ij} = \langle \bsigma_i, \bsigma_j \rangle_{H_r},$$
and symmetric operator $\bSigma: \calE_{\mathbf{o}} \to \calE_{\mathbf{o}}$ by $\bSigma \bsigma_j = \sum_{i=1}^n \Sigma_{ij} \bsigma_i$. We can then rewrite the equation for $\mathbf{x}_t$ as
$$d\mathbf{x}_t =\ptr_t \bSigma^{1/2} \circ d\mathbf{B}_t + \mathbf{a}_t(\mathbf{x}_t) \, dt,$$
with $\mathbf{B}_t$ being a Brownian in $(\calE_o, \langle \cdot , \cdot \rangle_{\brho})$. We then have the following result.

For any element $w \in T_xM$, $x\in M$, define $Y_w \in \mathfrak{g}$ by
$$\langle Y_w,Y \rangle_{H^r} = \langle w, Y(x)\rangle, \qquad Y \in \mathfrak{g} .$$
In other words $Y_w = L^{-1} \delta_x \otimes w$. If $\mathbf{w} = (w_1,\dots, w_n) \in T_{x_1} M \times \cdots \times T_{x_n}M$, we can define $Y_{\mathbf{w}} = L^{-1} \left(\sum_{j=1}^J \delta_{x_i} \otimes w_i \right)$. With this notation in place, we note that
$$\mathfrak{p} = \Big\{ Y_\mathbf{w} \, : \,  \mathbf{w} = (w_1,\dots, w_n) \in  T_{\mathbf{o}} \calM \Big\}$$
is the orthogonal complement to $\mathfrak{k} = \mathfrak{p}^\perp = \ker \bpi_{*,\id}$. We see that $\mathfrak{k}$ is the subalgebra of vector fields vanishing at $o_1, \dots, o_n$.

Let $\ad(A_1)^\dagger$ be the adjoint of the map $A_2 \mapsto \ad(A_1) A_2 = - [A_1, A_2]$ with respect to the inner product $\langle \cdot , \cdot \rangle_{H^r}$. Applying the results in Section~\ref{sec:ExHom} the setting of landmarks seen as a homogeneous space with a transitive action of diffeomorphisms, we have the following result.
\begin{theorem} Let
$$z_t = \sum_{j=1}^J c_{j,t} \sigma_{j} +a_t,$$ be a time-dependent vector field on $M$, and define a curve $\bgamma_t = (\gamma_{1,t}, \dots, \gamma_{n,t})$ along its directions by
    $$\dot \gamma_{r,t} = z_t(\gamma_{r,t}), \qquad \gamma_{r,0} = o_r, \qquad r=1, \dots, n.$$
    Let $A_t \in \mathfrak{g}$ be some solution of
    \begin{equation}
        \label{LandmarkUp} \dot A_t = - \ad(z_t)^\dagger A_t, \qquad A_0 = Y_{\mathbf{w}}, \mathbf{w} \in T_{\mathbf{o}}\calM.
    \end{equation}
\begin{enumerate}[\rm (a)]
    \item $\bgamma$ is singular if and only if there is a solution $A_t$ of~\eqref{LandmarkUp} satisfying
    $$\langle A_t, \sigma_j \rangle_{H_r} =0, \qquad j=1, \dots, J.$$
    \item If $\bgamma$ is a most probable path and not singular, then there is a solution $A_t$ of \eqref{LandmarkUp} with
    $$\langle A_t, \sigma_j \rangle_{H_r} =c_{t,r}, \qquad j=1, \dots, J.$$
\end{enumerate}
\end{theorem}
In other words, if we extend the definition of $\Sigma$ from \eqref{SigmaDiff} to be valid for arbitrary vector fields, then for most probable paths, we want to solve
$$\dot A_t = - \ad(\Sigma A_t + a_t)^\dagger A_t, \qquad A_0 = Y_{\mathbf{w}}, \mathbf{w} \in T_{\mathbf{o}}\calM.$$

\begin{proof}

Let $(\gamma_t, \alpha_t)$ be a solution of either \eqref{LieSingular} or \eqref{LieMPP} such that $\alpha_0$ vanishes on $\mathfrak{k}$, then by the Riesz representation theorem, there is a vector field $A_t \in \mathfrak{g}$, such that $A_0 \in \mathfrak{p}$. The result for singular curves now follows.

Finally, if $\bgamma$ is not singular and a dMPP, then
$$\Sigma \sharp^\rho \alpha_t = \sum_{j=1}^J \langle A_t , \sigma_j \rangle_{H^r} \sigma_j = z- a_t = \sum_{j=1}^J c_{j,t} \sigma_{j},$$
which gives us the result.
\end{proof}

We can also compute these vector fields from the point of view of $\calM$.
\begin{theorem} \label{th:LMDown}
Let $\mathbf{o} =(o_1,\dots, o_n) \in \calM$ be a given point in the landmark manifold. For a given Brownian motion $\hat B_t$ in $\mathbb{R}^J$ and vector fields $\sigma_{1}, \dots, \sigma_J, a_t$, define a stochastic process $\mathbf{x}_t = (x_{1,t}, \dots, x_{n,t})$ by $dx_t = \sum_{j=1}^J \bsigma_j(\mathbf{x}_t) \circ d\hat B^j + \mathbf{a}_t dt$, i.e.,
$$dx_{i,t} = \sum_{j=1}^J \sigma(x_{i,t}) \circ d\hat B^j + a_t(x_{i,t}) dt, \qquad i =1, \dots, n.$$
Let $\bgamma_{t} = (\gamma_{1,t}, \dots, \gamma_{n,t})$ be a curve in $\calM$ starting at $\mathbf{o}$ with $\dot \bgamma_t = \mathbf{z}_t(\bgamma_t)$, for some time-dependent vector field $z_t = \sum_{j=1}^J c_{j,t} \sigma_j + a_t$.
Along each curve $\gamma_{r,t}$, $r=1,\dots, n$, let $\lambda_{r,t}$ be a form satisfying
\begin{equation}
    \label{LandmarkDown}\frac{d}{dt} \lambda_{r,t}(W) - \lambda_{r,t}([z_t, W]) = 0, \qquad \text{for any $W \in \Gamma(TM)$.}
\end{equation}
\begin{enumerate}[\rm (a)]
    \item The curve $\bgamma_t$ is singular if and only if there are solutions of \eqref{LandmarkDown} such that 
    $$\sum_{r=1}^n \lambda_{r,t}(\sigma_j(\gamma_{r,t})) =0. \qquad \text{for any $j =1, \dots, J$.}$$
    \item If the curve $\bgamma_t$ is not singular and a developed most probable path then there are solutions of \eqref{LandmarkDown} such that for any $j=1, \dots,J$,
    $$\sum_{r=1}^n \lambda_{r,t}(\sigma_j(\gamma_{r,t})) = c_{j,t} $$
\end{enumerate}
\end{theorem}
We remark that because our connection here is flat, the above solutions will be usual sub-Riemannian geodesics in $\calM$ as explained in Section~\ref{sec:RSR}. Also note that even if $\sigma_1, \dots, \sigma_J$ span the entire tangent bundle $TM$, the problem is still sub-Riemannian when $J < nd = \dim \calM$.
\begin{proof}
Let $\bgamma_t$ be a curve in $\calM$ starting at $\mathbf{o}$ and let $\blambda_t$ be a one-form along this curve. Since the connection $\nabla$ is flat on $E$, we have $K =0$. Hence, both for singular curves and most probable paths, we need only consider the equation
$$\frac{D}{dt} \blambda_t + \blambda_t(\nabla \mathbf{a}_t) + \blambda_t(T(\dot \bgamma, \cdot )) =0.$$
We write
$$\dot \bgamma_t = \mathbf{z}_t(\bgamma_t)  = \sum_{j=1}^J c_{j,t} \bsigma_t(\bgamma_t) + \mathbf{a}_t(\bgamma_t).$$

Assume that locally, $\dot \bgamma_t$ is represented as $\dot \bgamma_t = \mathbf{z}_t(\bgamma_t) $. We then get for any vector field $\mathbf{W}$ on $M$, we have that
\begin{align*}
   &  \left(\frac{D}{dt} \blambda_t\right)(\mathbf{W}) + \blambda_t(\nabla_{\mathbf{W}} \mathbf{a}_t) + \blambda_t(T(\dot \bgamma, \mathbf{W} )) \\
   & =  \frac{d}{dt} \blambda_t(\mathbf{W})  - \blambda_t(\nabla_{\mathbf{z}_t } \mathbf{W}) + \blambda_t(\nabla_{\mathbf{W}} \mathbf{a}_t) + \blambda_t(\nabla_{\mathbf{z}_t} \mathbf{W} - \nabla_{\mathbf{W}} \mathbf{z}_t - [\mathbf{z}_t , \mathbf{W}] ) \\
   & =  \frac{d}{dt} \blambda_t(\mathbf{W}) - \blambda_t( [\mathbf{z}_t , \mathbf{W}] )
\end{align*}
Focusing on each component, we have the result.
\end{proof}

\section{Numerical experiments}
\label{sec:numerics}
We here numerically integrate and visualize development most probable paths as identified in the previous sections between points on three manifolds: The Lie group $\SO(3)$ with a left-invariant metric, the homogeneous space $\mathbb S^2$, and the LDDMM landmark manifold.\footnote{Figures and computations using Jax Geometry \url{https://bitbucket.org/stefansommer/jaxgeometry/}.}

\subsection{Lie groups and homogeneous spaces}
We illustrate the dynamics on Lie groups and homogeneous spaces using the rotation group $\SO(3)$ and the homogeneous space $\mathbb S^2\simeq \SO(3)/\SO(2)$. The rotations are illustrated by the action on the standard, orthonormal frame $e_1,e_2,e_3\in\mathbb R^3$. The projection of $g\in SO(3)$ to $\mathbb S^2$ is equivalent to the mapping $g\mapsto g.e_1$.

Figure~\ref{fig:SO3} shows most probable paths in $\SO(3)$ equipped with a left-invariant connection. $\Sigma_t=\mathrm{diag}(0.3,2,1)$ in the standard basis for $\mathfrak{so}(3)$, and the drift is constant $a_t=(1,0,0)$. The figure illustrate both the forward solution of the system \eqref{LieMPP} from the identity $e\in\SO(3)$ with specified initial conditions $\alpha_0$, and the boundary value problem of a most probable paths between two points $e,g\in\SO(3)$.
\begin{figure}[ht]
    \centering
    \begin{subfigure}[b]{\textwidth}
        \centering
        \includegraphics[trim=200 50 200 50,clip,width=0.32\linewidth]{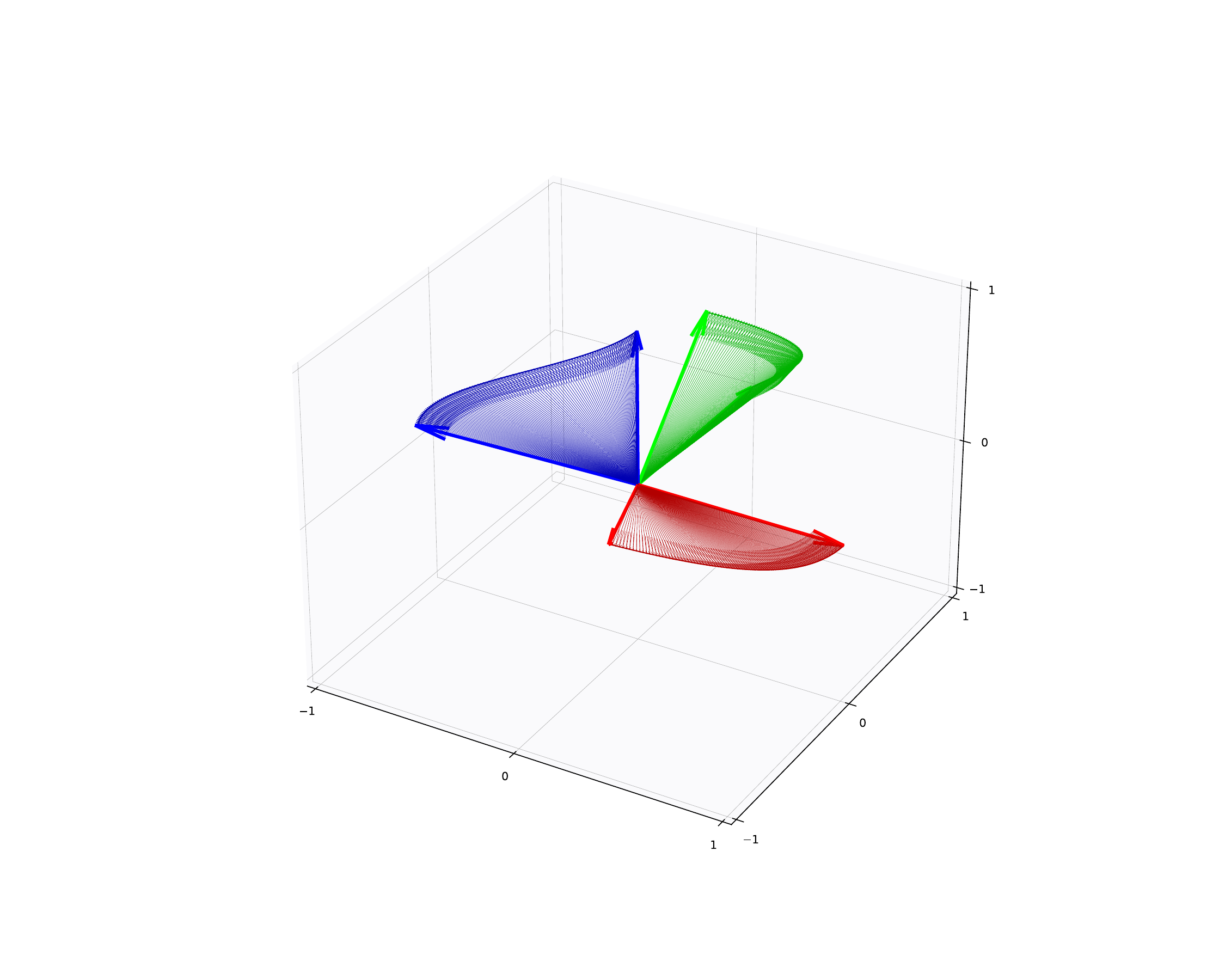}
        \hspace{1cm}
        \includegraphics[trim=200 50 200 50,clip,width=0.32\linewidth]{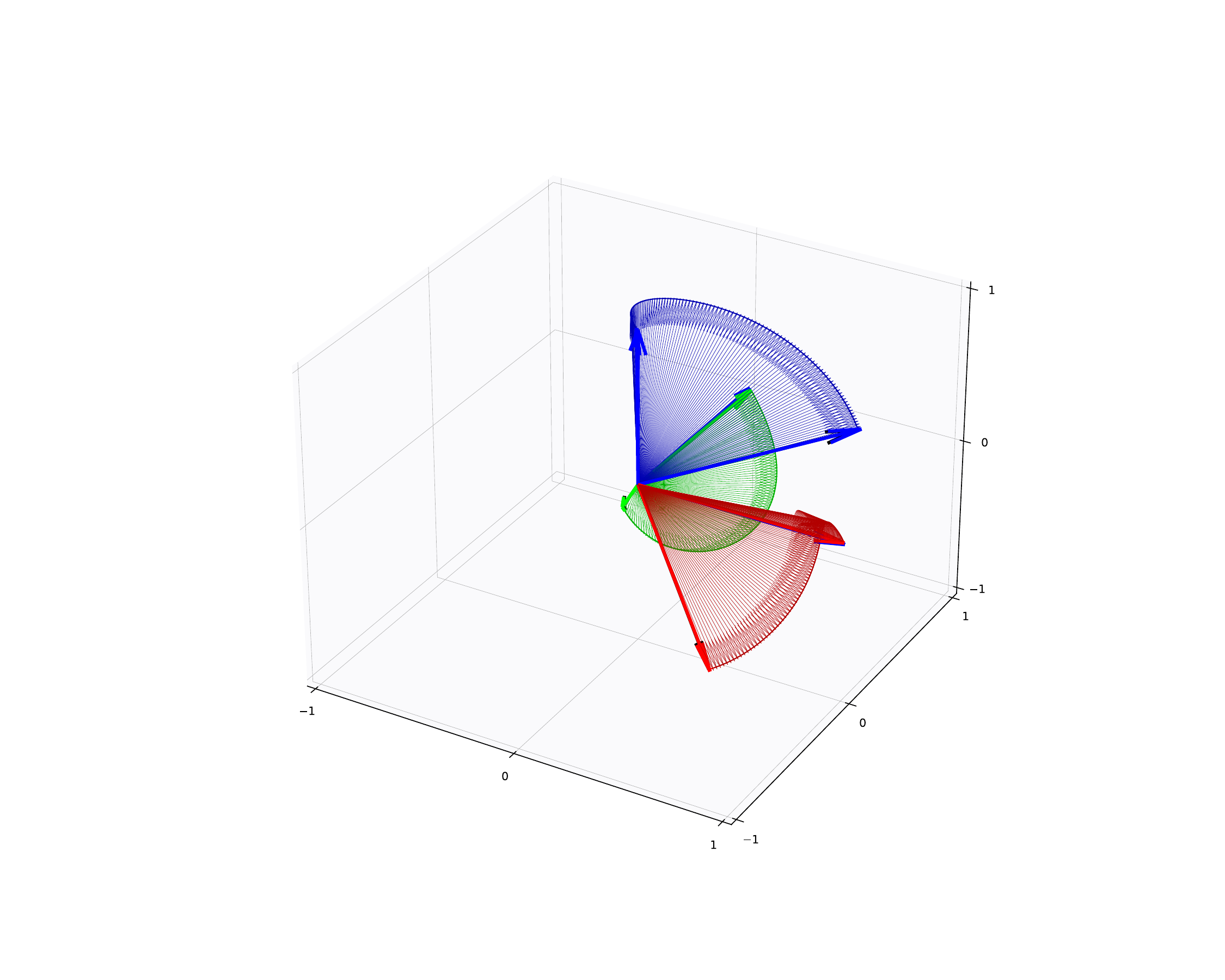}
    \end{subfigure}
    \caption{Examples of most probable paths on $\SO(3)$, illustrated by the action on the standard basis of $\mathbb R^3$ (red, green, blue vectors): (left) Initial value problem of most probable path starting at the identity $e\in\SO(3)$ and evolving from initial conditions $\alpha_0$; (right) solution of boundary value problem of finding most probable paths between two points $e,g\in\SO(3)$.}
    \label{fig:SO3}
\end{figure}
In Figure~\ref{fig:S2}, we find most probable path on the sphere $\mathbb S^2$ as projection of paths on the top space $\SO(3)$ of which $\mathbb S^2$ is a quotient. We set $\Sigma_t=\id_3$ but add non-trivial drift in the form of a vector field on $\mathbb S^2$. The field is horizontally lifted to $\SO(3)$ when solving the most probable paths equations \eqref{HomoGenMPP}. The figure shows the drift field, the solution of a boundary value problem and the corresponding paths on $\SO(3)$.
\begin{figure}[ht]
    \centering
    \begin{subfigure}[b]{\textwidth}
        \centering
        \includegraphics[trim=200 50 200 50,clip,width=0.32\linewidth]{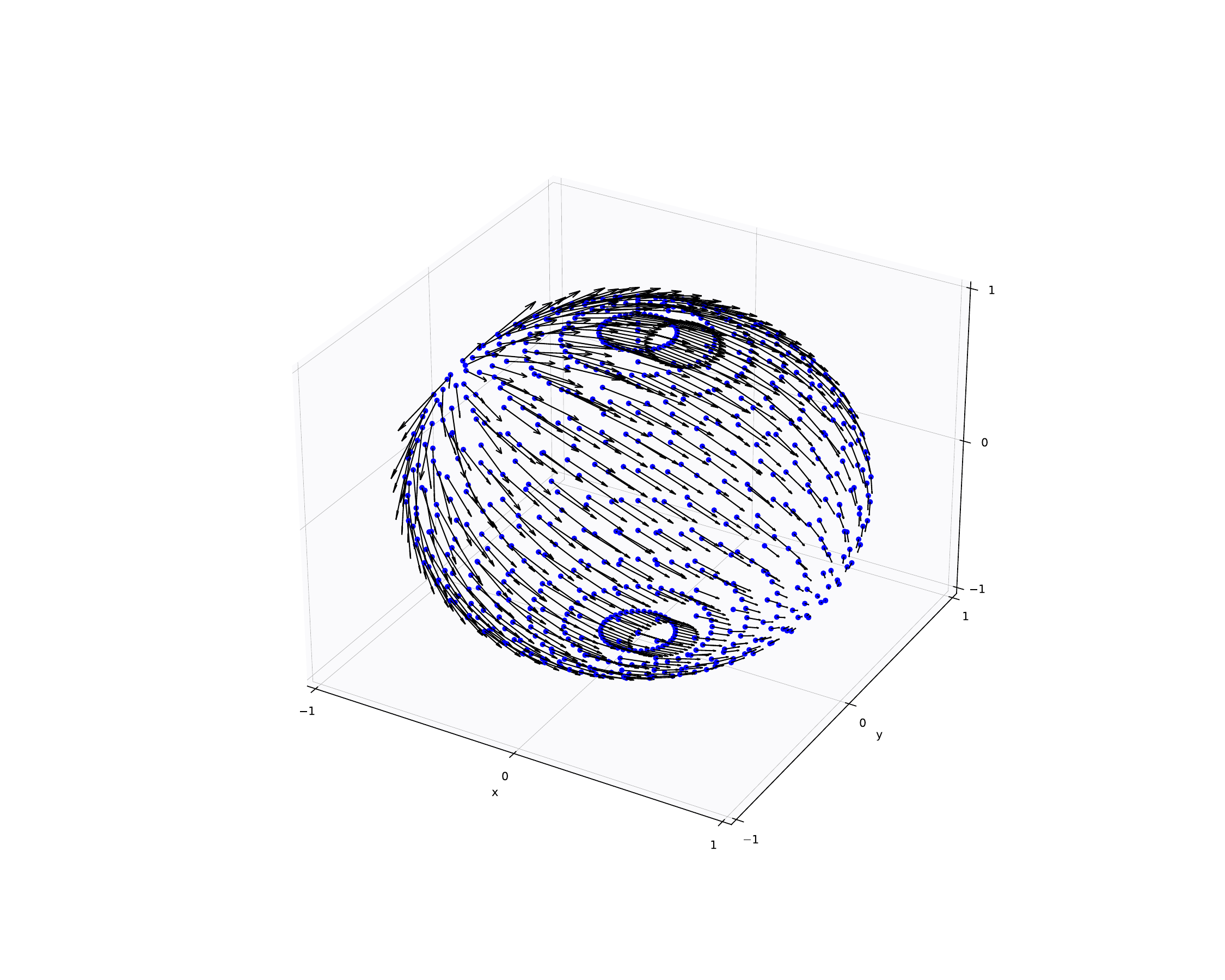}
        \includegraphics[trim=200 50 200 50,clip,width=0.32\linewidth]{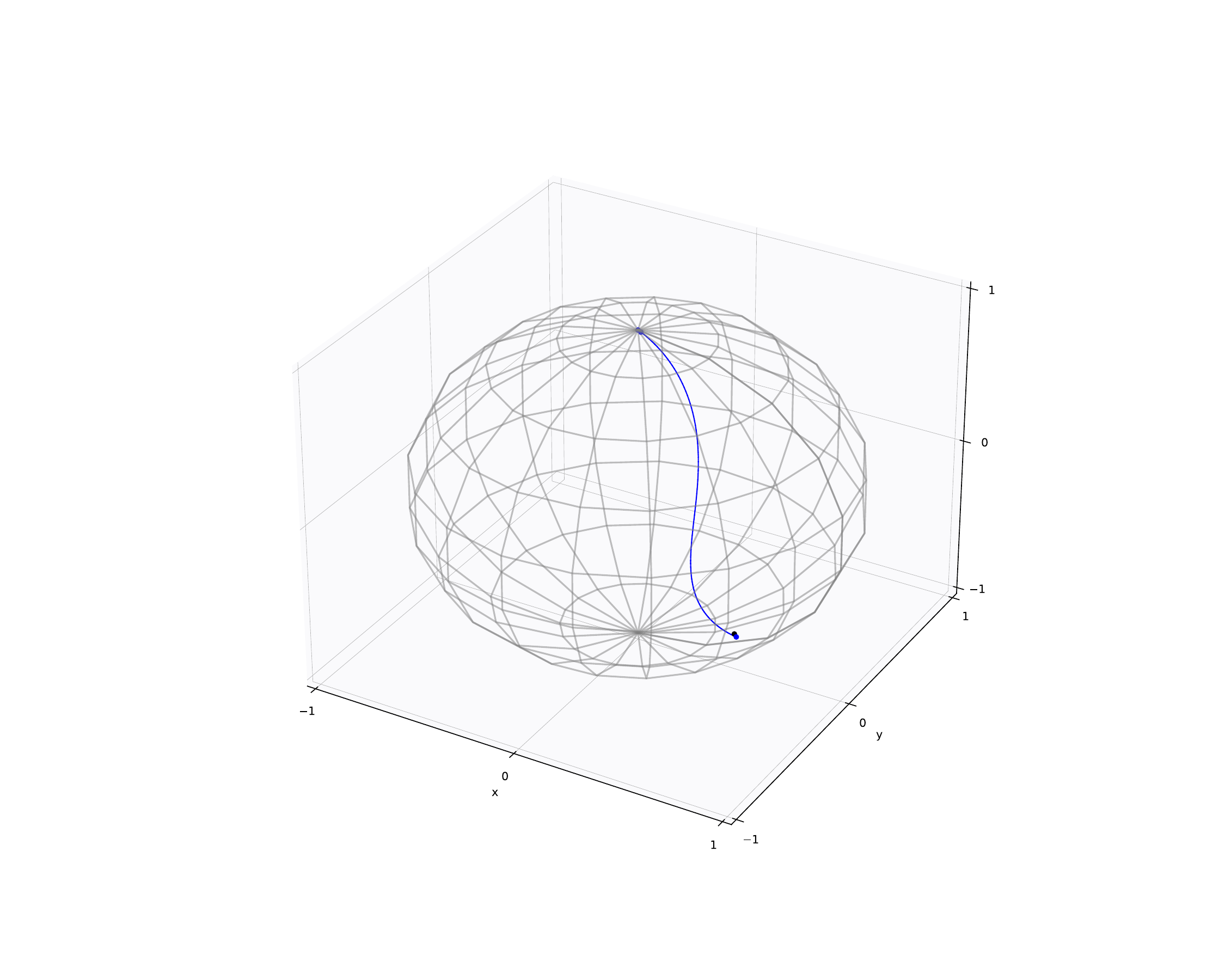}
        \includegraphics[trim=200 50 200 50,clip,width=0.32\linewidth]{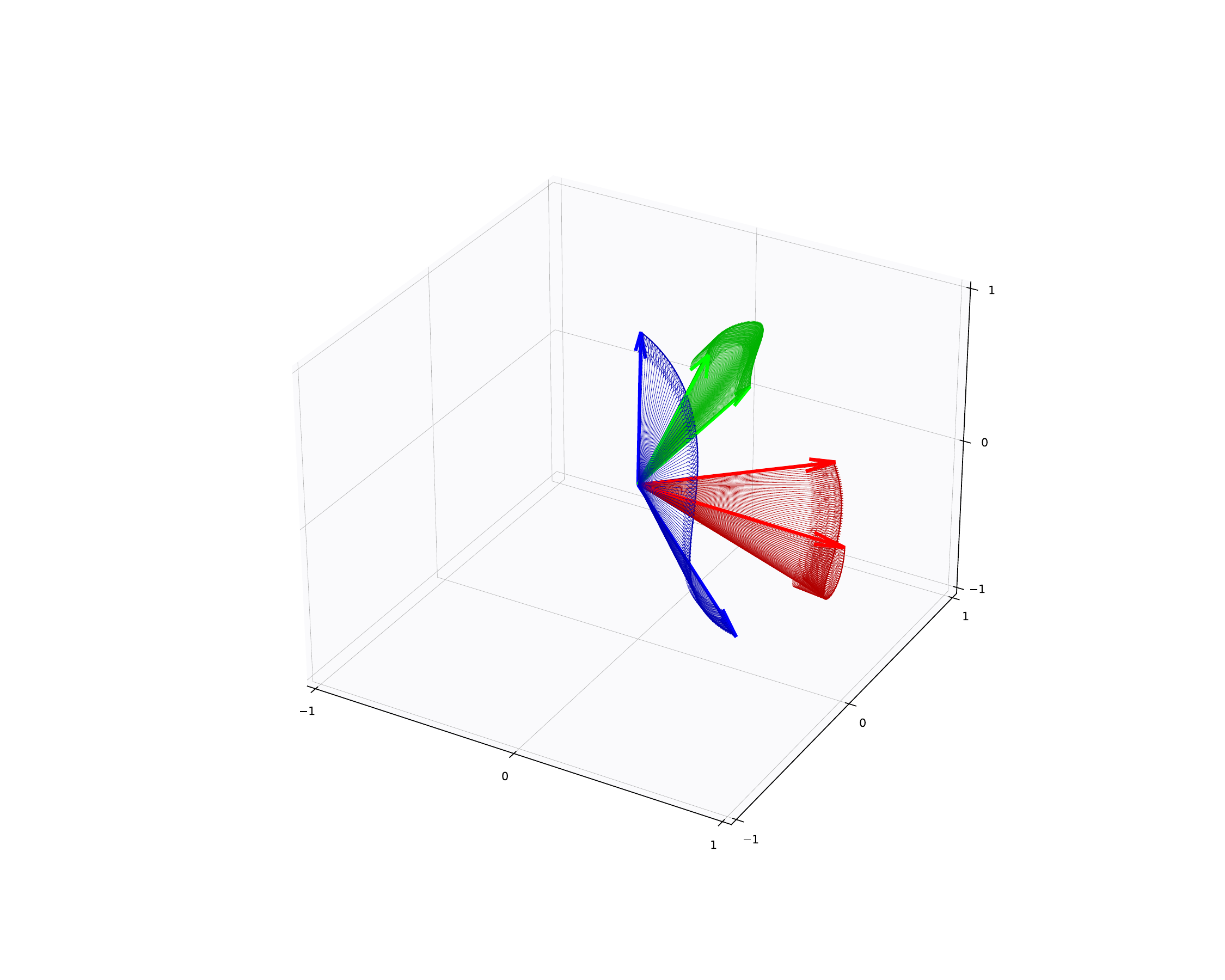}
    \end{subfigure}
    \caption{Examples of most probable paths on $\mathbb S^2$: (left) Drift field $a$ on $\mathbb S^2$ (scaled arrows); (center) solution of boundary value problem of finding most probable paths between two points on $\mathbb S^2$; (right) path in $\SO(3)$ generating the most probable paths between the points on $\mathbb S^2$.
    }
    \label{fig:S2}
\end{figure}

\subsection{Landmark manifolds}
We here illustrate most probable paths of diffeomorphisms acting on landmark configurations. We visualize the underlying deformations by their action on the initial landmark configuration for each $t$. The evolution is governed by equations of Theorem \ref{th:LMDown}, given in coordinates in appendix \ref{sec:landmarks_coordinates}. Figure~\ref{fig:landmarks}, left column shows the action on the landmark configuration of a most probable path of diffeomorphism, while, for comparison, Figure~\ref{fig:landmarks}, right column shows most probable transformation between the same configurations following \cite{GrSo22}. The results are shown for two different cases of noise fields (see figure caption). As discussed in remark~\ref{rem:GrSo22}, the most probable paths are paths of diffeomorphisms while the most probable transformations results in the flow lines of the landmarks being most probable.
\begin{figure}[ht!]
    \centering
    \begin{subfigure}[b]{\textwidth}
        \centering
        \includegraphics[trim=100 50 100 50,clip,width=0.4\linewidth]{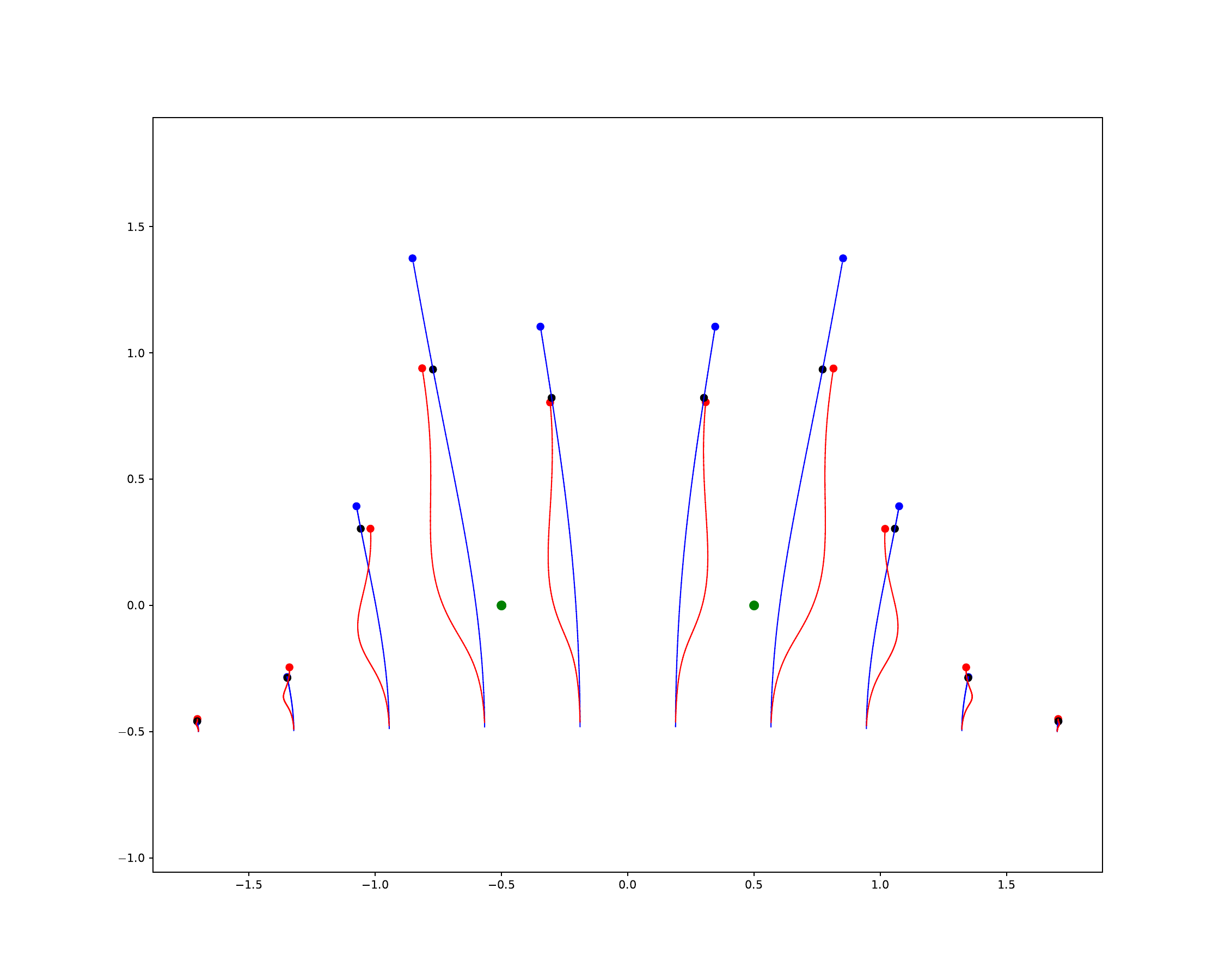}
        \hspace{1cm}
        \includegraphics[trim=100 50 100 50,clip,width=0.4\linewidth]{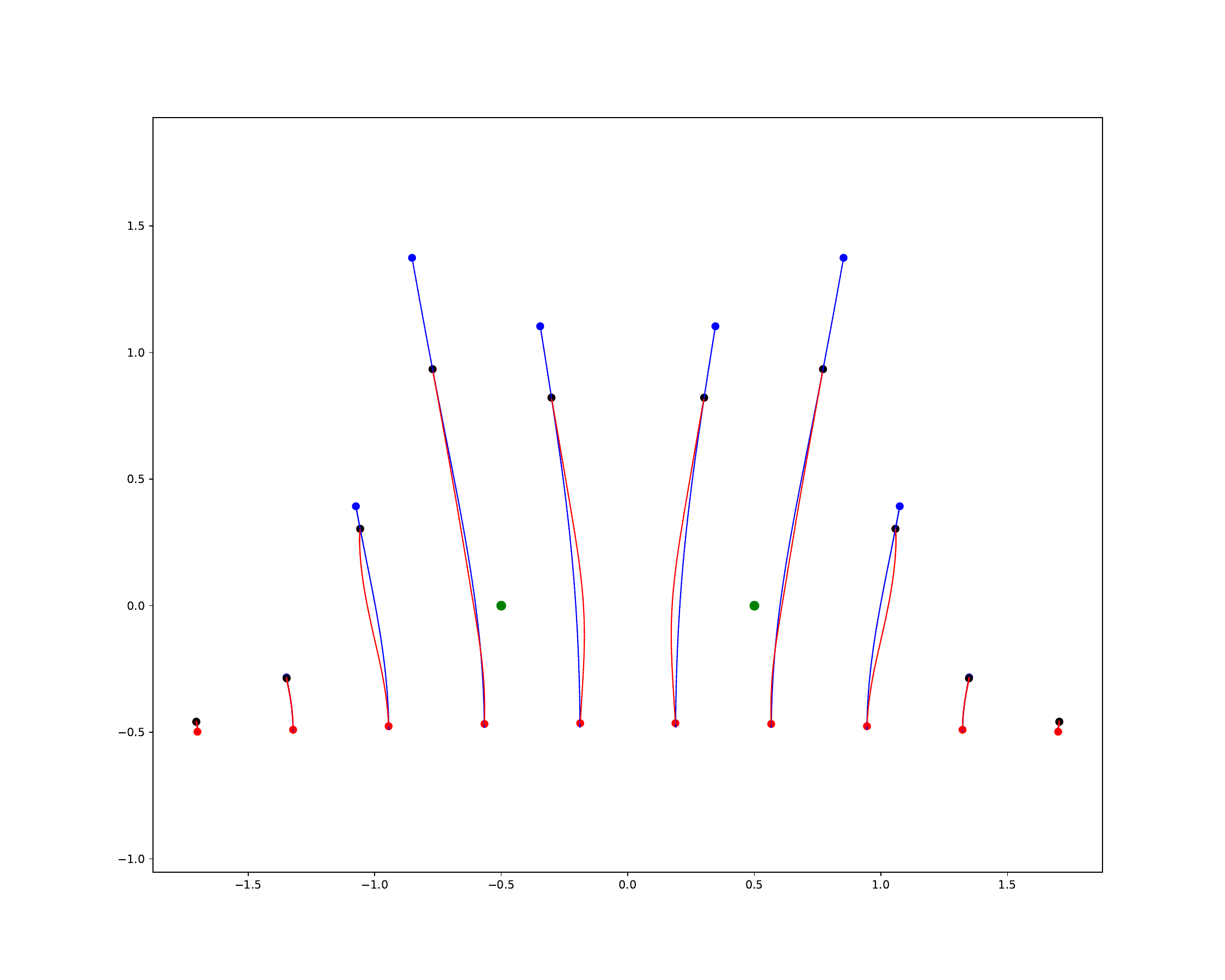}
        \\
        \includegraphics[trim=100 50 100 50,clip,width=0.4\linewidth]{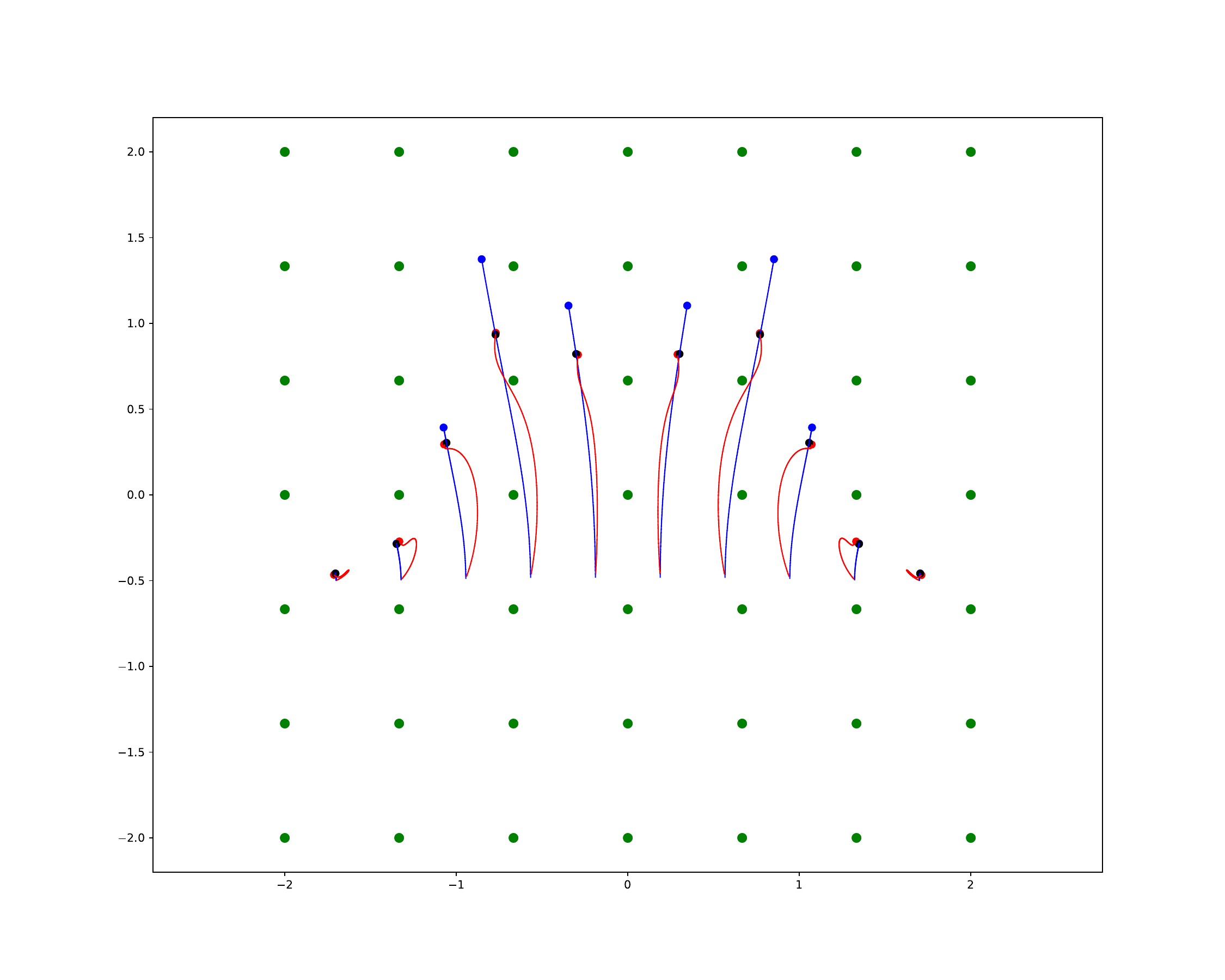}
        \hspace{1cm}
        \includegraphics[trim=100 50 100 50,clip,width=0.4\linewidth]{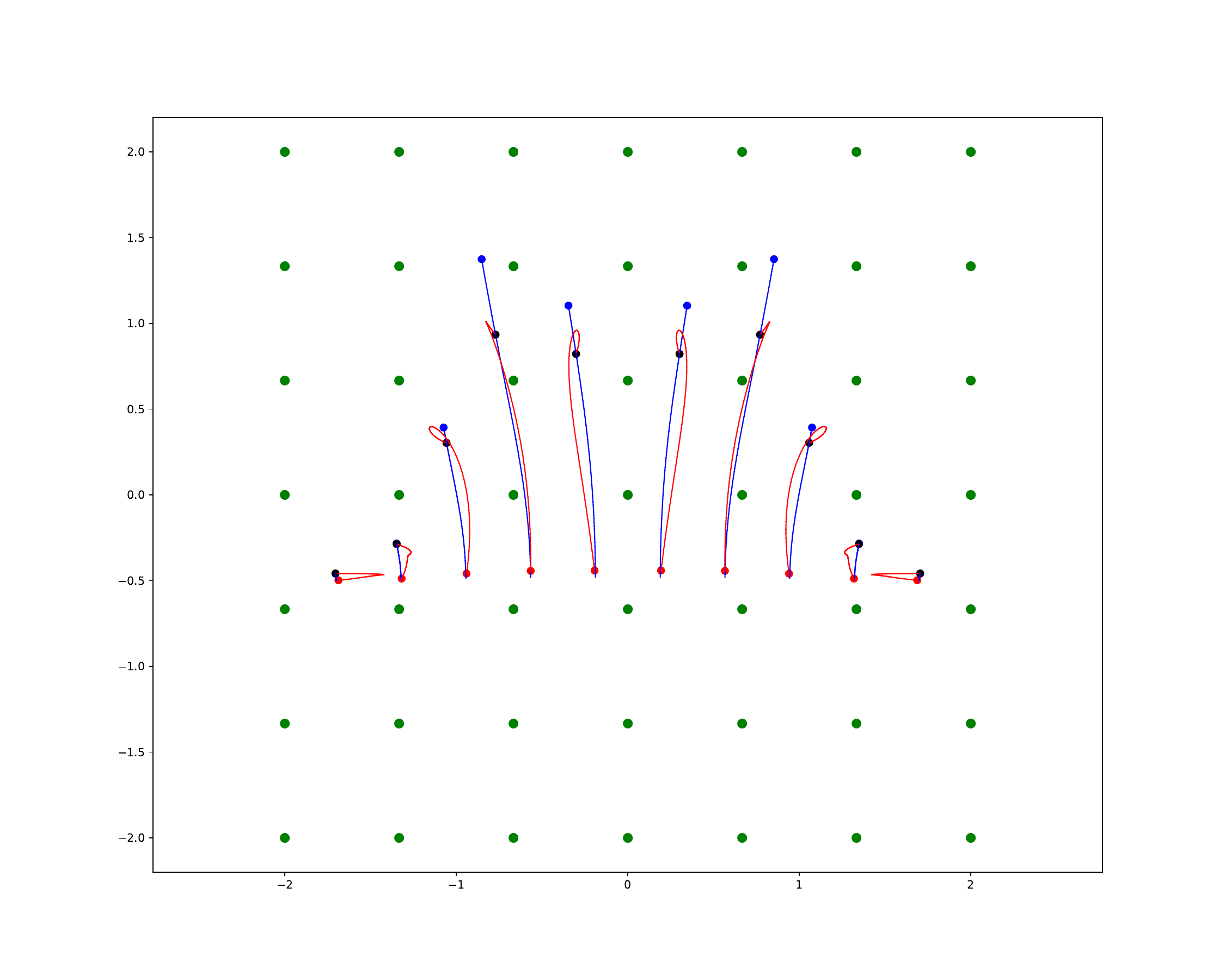}
    \end{subfigure}
    \caption{Most probable paths between two landmark configuration. Blue curves: Trajectories with no noise, only deterministic drift $a_t$. Red curves: Most probable paths to target configuration (black landmarks). Left column: Development most probable path as identified in Theorem~\ref{th:LMDown}. Right column: Most probable transformation \cite{GrSo22} that results in most probable flow lines. Noise fields $\sigma$ indicated by green dots. Top row: two noise fields. Bottom row: 7x7 grid of noise fields.}
    \label{fig:landmarks}
\end{figure}
In Figure~\ref{fig:landmarks2}, we perform the same experiments but now with configurations with 128 landmarks and $13\times 13$ grid of noise fields.
In both figures, particularly Figure~\ref{fig:landmarks2}, we see that the development most probable paths make very regular transformations between the landmark configurations, while the most probable transformations results in landmark trajectories with higher curvature, and the landmarks tend to overshoot the target requiring a hard turn to hit the target. However, it is hard to give general qualitative descriptions of the differences between most probable paths and most probable transformations.
\begin{figure}[ht!]
    \centering
    \begin{subfigure}[b]{\textwidth}
        \centering
        \includegraphics[trim=100 50 100 50,clip,width=0.45\linewidth]{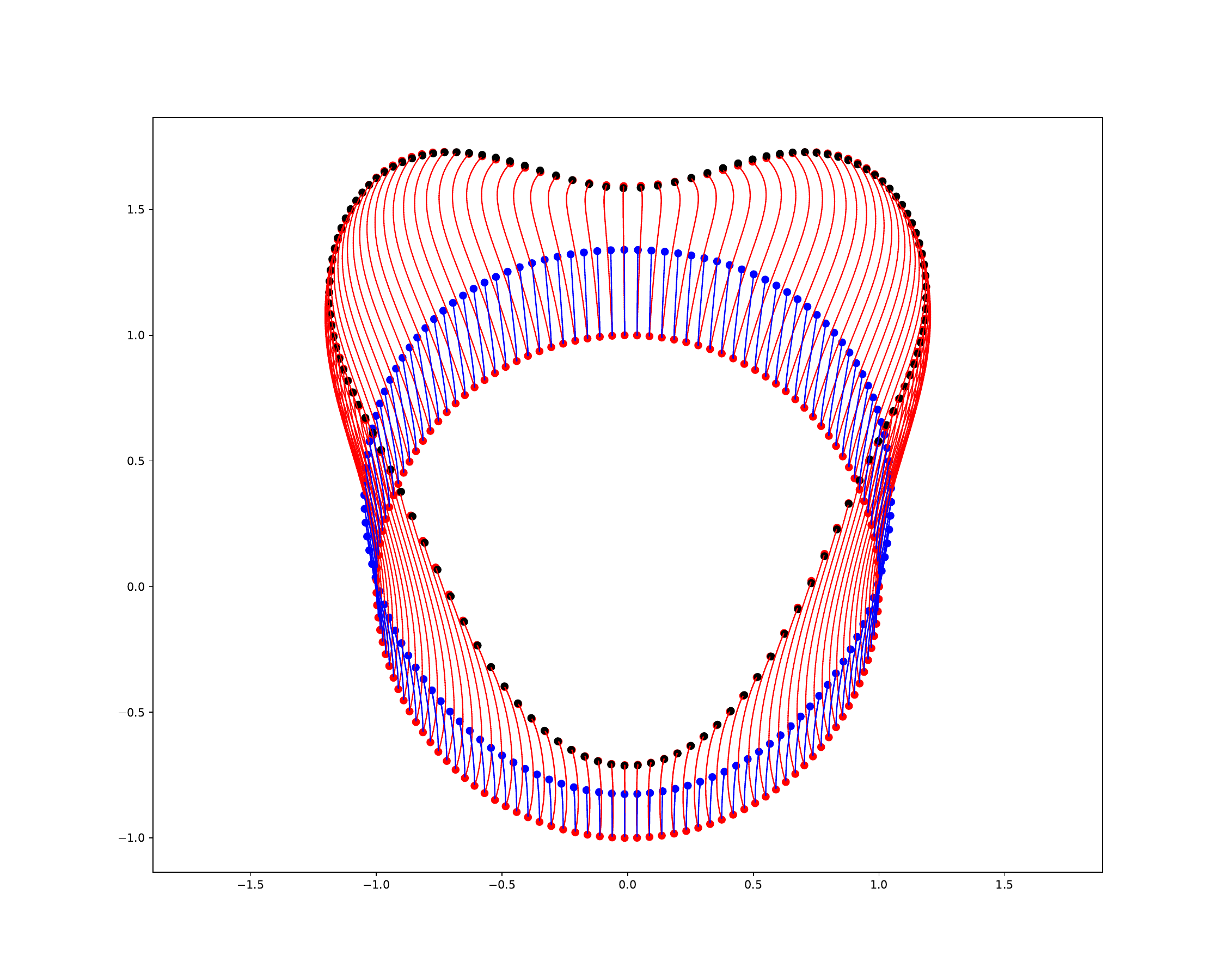}
        \hspace{1cm}
        \includegraphics[trim=100 50 100 50,clip,width=0.45\linewidth]{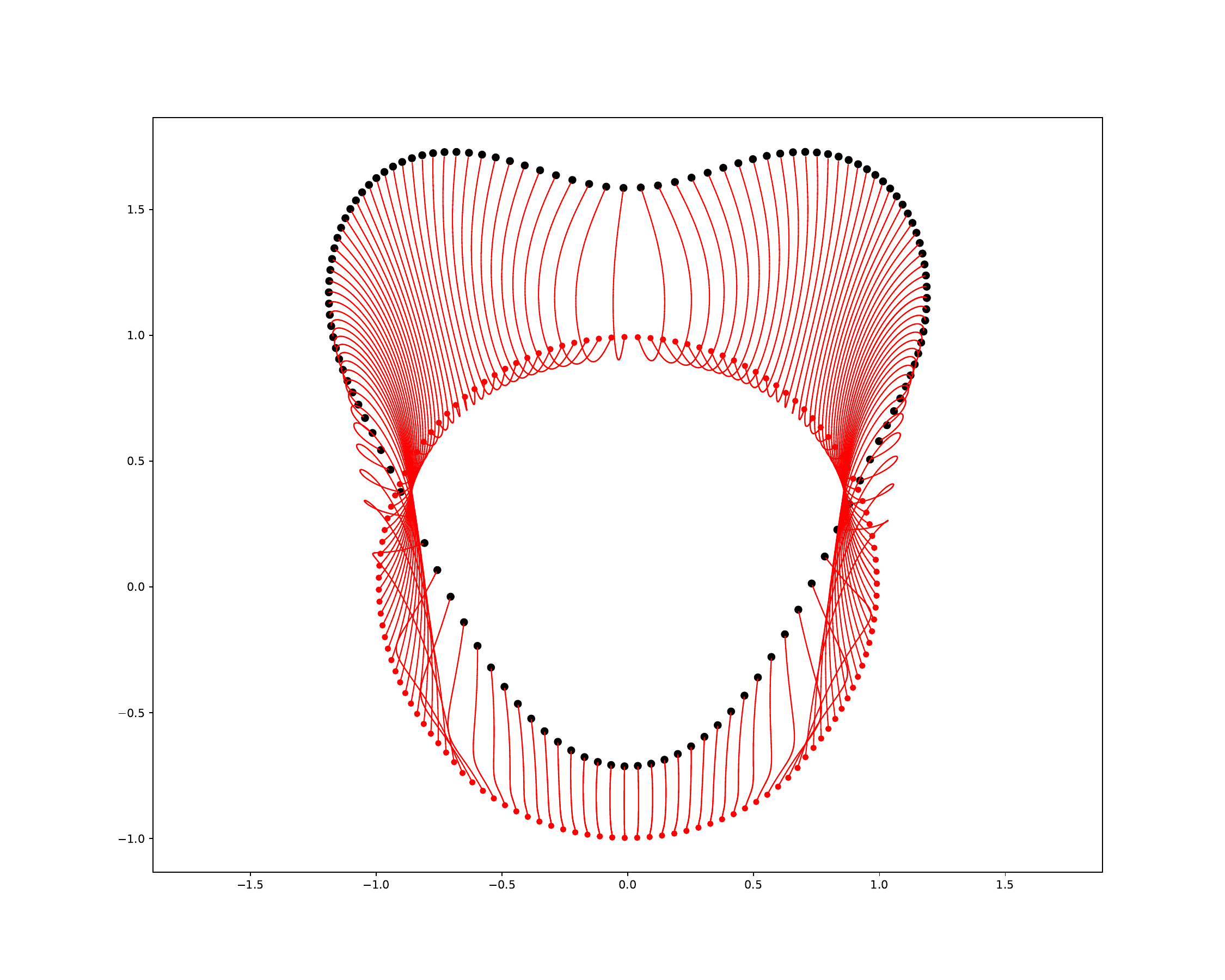}
    \end{subfigure}
    \caption{Left: Most probable paths between two landmark configurations consisting of 128 points. 
Blue curves: Trajectories with no noise.
        Right: corresponding most probable transformation \cite{GrSo22}
    }
        \label{fig:landmarks2}
\end{figure}

\newpage
\appendix
\section{Developed Most probable paths in local coordinates}
In all of the examples below, the manifold has dimension $d$ and summation over $d$ entries are in Latin letters. We assume that the sub-Riemannian structure has rank $J$, and summation over $J$ entries are in Greek letters.

\subsection{General formulas for most probable paths}
We want to give local description for the equation \eqref{MPPu}. Let $\gamma_t$ be a given curve starting at $o$. Let $(x^1, \dots, x^d)$ be a set of local coordinates on $M$ and define $x^j_t =x^j(\gamma_t)$ and $u_t = \sum_{j=1}^d u^j_t \partial_j$. Write $\Gamma_{ij}^k$ for the Christoffel symbols of $\nabla$ in these coordinates, and write $T_{ij}^k$ and $R^{l}_{ijk}$ for local representation of respectively the torsion and the curvature. Let the sub-Riemannian structure $(E,g)$ have rank $J$.
Choose a basis $f_{1,0} ,\dots, f_{d,0}$ of $T_o M$ such that $f_{1,0}, \dots, f_{J,0}$ is an orthonormal basis of $E_o$. Write $\Sigma_t f_{\beta,0} = \sum_{\alpha=1}^J \Sigma_{\alpha\beta,t} f_{\alpha,0}$.
Define $f_{1,t} ,\dots, f_{d,t}$ by parallel transport along $\gamma$. In other words, if we write $f_{r,t} = \sum_{i=1}^d f_{ir,t} \partial_i$, then they are solutions of
$$\dot f_{kr,t} = -\sum_{i,j,k=1}^d \dot x_i f_{jr,t} \Gamma_{ij}^k.$$
Let $f_t^1, \dots, f_t^{d}$ be the corresponding coframe and write $f_t^r = \sum_{j=1}^d f^{rj}_t dx^j$, meaning that $(f_t^{ij})$ is the inverse matrix of $(f_{ij,t})$.
Writing $\lambda_t = \sum_{i=1}^d \lambda_{i,t} f_t^i$ and $\chi_t = \frac{1}{2} \sum_{\alpha,\beta=1}^J \chi^{\alpha\beta}_t f_{\alpha,t} \wedge f_{\beta,t}$, then by evaluation of \eqref{MPPu} along $f_{r,t}$, we obtain
\begin{align*}
\dot \lambda_{r,t} & = - \sum_{i,k,l=1}^d \lambda_{l,t} f^{lk}_t f_{ir,t} \left(\partial_i u^k_t + \sum_{j=1}^d u_{t}^j \Gamma_{ij}^k  + \sum_{j=1}^d \dot x_j T_{ij}^k  \right) \\
& \qquad +  \sum_{i,j,k,l=1}^d \sum_{\alpha, \beta=1}^J \dot x_i \chi^{\alpha\beta} f_{jr,t} f_{k\alpha,t} f_{l\beta,t} R_{ijk}^l ,\\
\dot \chi_t^{\alpha\beta} & = \frac{1}{2} \sum_{\nu=1}^J \lambda_{\nu,t} \left(   \Sigma_{\alpha \nu,t}  \lambda_{\beta,t} -  \Sigma_{\beta \nu,t} \lambda_{\alpha,t} \right), \qquad \chi_T = 0 \\
\dot x_t^k & = u_t^k+\sum_{\alpha,\beta =1}^J \Sigma_{\alpha\beta,t} \lambda_{\beta,t} f_{k\alpha,t}.
\end{align*}

\subsection{Most probable paths on Lie groups}
We give a local description of the equations \eqref{LieMPP}. We can solve the equation in the Lie algebra. Let $A_1, \dots, A_d$ be a basis for $\mathfrak{g}$ such that $A_1, \dots, A_J$ is an orthonormal basis of $\mathfrak{e}$. Write their Lie brackets as $[A_i,A_j] = \sum_{k=1} c_{ij}^k A_k$. Write $\Sigma_t A_\nu = \sum_{\mu=1}^J \Sigma_{\mu \nu,t}$ and let $\alpha_{j,t} = \alpha_t(A_j)$. Then most probable paths are solutions to equations

$$\left\{ \begin{aligned}
\dot \alpha_{j,t} &= - \sum_{k=1}^d \alpha_{k,t} \left(\sum_{\mu,\nu=1}^J \Sigma_{\mu\nu} \alpha_{\nu,t} c^k_{\mu j} + \sum_{i=1}^d a_t^i c_{ij}^k\right) \\
z_t&  = \sum_{\mu,\nu=1}^J \Sigma_{\mu\nu} \alpha_{\mu,t} A_{\nu} + a_t \\
\dot \gamma &= z_t \cdot \gamma_t
\end{aligned} \right.
$$

\subsection{Most probable paths for landmarks}
\label{sec:landmarks_coordinates}
We consider the equations for most probable paths for landmarks, described in Theorem~\ref{th:LMDown}~(b). We consider a curve $\bgamma_t = (\gamma_{1,t}, \dots, \gamma_{n,t})$, and let $\lambda_{r,t}$ be a solution of \eqref{LandmarkDown}. Let $(x^1, \dots, x^d)$ be the local coordinate system on $M$, and define
$$\lambda_{r,t} = \sum_{k=1}^d \lambda_{r,t}^k dx^k,
\qquad \sigma_\alpha = \sum_{k=1}^d \sigma_{\alpha}^k \partial_k, \qquad a_t = \sum_{k=1}^d a_{t}^k \partial_k,$$
and $x^k(\gamma_{r,t}) = x_{r,t}^k$.
Then we can solve
\begin{align*}
\dot \lambda_{r,t}^k & = -\sum_{i=1}^d \sum_{\alpha=1}^J \lambda_{r,t}^i c_{\alpha,t} \partial_k \sigma_{\alpha}^i(\gamma_{r,t})- \sum_{i=1}^d \lambda_{r,t}^i \partial_k a_{t}^i(\gamma_{r,t}), \\
c_{\alpha,t} & = \sum_{i,r=1}^n  \lambda_{r,t}^i \sigma_{\alpha}^i(\gamma_{r,t})  , \\
\dot x_{r,t}^k & = a_{t}^k+ \sum_{j=1}^J c_{\alpha,t} \sigma_{\alpha,t}^k(\gamma_{r,t}).
\end{align*}

\section{Proof of Lemma~\ref{lemma:AffineDev} for absolutely continuous curves} \label{sec:AbsCont}
We will give some details on an alternative proof of Lemma~\ref{lemma:AffineDev}, which makes it more clear that it is still holds for absolutely continuous curves. In order to complete this, we need to define some additional structures on the general frame bundle, which are not used other places of the paper. For proofs of these statements, see e.g. \cite[Section~2.2 and 3.2]{Gro22}.

Let $\GL(d) \to \GL(TM) \stackrel{\pi}{\to} M$ be the general frame bundle over $M$. That is $\GL(TM)_x$ consist of all linear invertible maps $f = (f_1, \dots, f_n):\mathbb{R}^d \to T_x M$, with $\GL(d)$ acting on the right by precomposition. For $f \in \GL(d)$, we define functions $\bar{T}: \GL(TM) \to \wedge^2 (\mathbb{R}^d)^* \otimes \mathbb{R}^d$ and $\bar{R}: \GL(TM) \to \wedge^2 (\mathbb{R}^d)^* \otimes (\mathbb{R}^d)^* \otimes \mathbb{R}^d$ by $\bar{T}(f) = f^{-1} T(f \cdot, f \cdot)$ and $\bar{R}(f) = f^{-1} R(f \cdot, f \cdot) f$.

Consider $f_t = (f_{1,t}, \dots, f_{d,t})$ as a $\nabla$-parallel frame along $\gamma_t$ with $f_0 = f$ and $\dot \gamma_0 =v$. Similar to Section~\ref{sec:OptimalControl}, we define $\dot f_t |_{t=0} = h_{f_0} v$. We also write $H_j(f_0) = h_{f_0} f_{j,0}$, giving us vector fields $H_1, \dots, H_d$ on $\GL(TM)$. If $a = (a^1, \dots, a^d) \in \mathbb{R}^d$, we define $H_a = \sum_{i=1}^d a^i H_i$. Furthermore, for any $A \in \mathfrak{gl}(d)$, define a vector field $\xi_A$ by
$$\xi_A|_f = \frac{d}{dt} f \cdot e^{t A}|_{t=0}.$$
We then have Lie brackets $[H_a, H_b] = - H_{\bar{T}(a,b)} - \xi_{\bar{R}(a,b)}$, $[\xi_A, H_a]= H_{Aa}$ and $[\xi_A, \xi_B]= \xi_{[A,B]}$. If we define one forms $\eta$ and $\theta$ with values in respectively $\mathfrak{gl}(d)$ and $\mathbb{R}^d$, by
$$\eta(\xi_A + H_a) =A, \qquad \theta(\xi_A + H_a) = a,$$ 
we then have relations $\Omega = d\eta + \frac{1}{2} [\eta,\eta]$ and $\Theta = d\theta + [\eta, \theta]$, where
$$\Omega(H_a + \xi_A, H_b + \xi_B) = \bar{R}(a,b), \qquad \Theta(H_a + \xi_A, H_b + \xi_B) = \bar{T}(a,b),$$
with $A,B \in \mathfrak{gl}(d), a,b \in \mathbb{R}^d$. This gives us the following result.
\begin{lemma}
If $\omega, k \in \mathbb{H}^T(T_oM)$, then $\Dev(\omega + sk)$ is well defined and Lemma~\ref{lemma:AffineDev} holds.
\end{lemma}

\begin{proof}
If $\omega, k \in \mathbb{H}^T(TM)$, then $\dot \omega + s \dot  k_t$ is in $L^2$, and we can define $F:(-\ve, \ve) \times [0,T] \to \GL(M)$ as the solution of $F(s,0) = f_0 \in \GL(TM)_o$ and
$$\partial_t F(s,t) = H_{f_0^{-1}(\dot \omega_{j,t} + s \dot k_{j,t})}(F(s,t)) + hu_t(F(s,t)).$$
Write $f_t = F(0,t)$ and
$$\partial_s F(s, t) |_{s=0} = H_{f_0^{-1} v_t} + \xi_{f_0^{-1} A_t f_0}.$$
with $v_t$ and $A_t$ being curves in respectively $T_oM$ and $T^*_o M \otimes T_o M$ that vanish at time $t=0$.
Then $\pi(F(t,s)) = \gamma_t^s = \Dev_u(\omega + sk)_t$, $\ptr_t = f_t \circ f_0^{-1}$ and $\partial_s \gamma_t^s |_{s =0} = \ptr_t v_t$. We furthermore see that
\begin{align*}
    & (F^* \Omega)(\partial_s, \partial_t) |_{s=0} = \bar{R}_{f_t}(f_t^{-1} v_t, f_t^{-1} (\dot \omega + u_{t, \ptr_t}))  = f_0^{-1} R_{\ptr_t}(v_t, \dot \omega + u_{t, \ptr_t}) \\
    & = (F^* d\eta)(\partial_s, \partial_t)|_{s=0} + [(F^* \eta)(\partial_s), (F^* \eta)(\partial_t)]|_{s=0} - [(F^*\eta)(\partial_t),F]\\
    & = - \partial_t F^*\eta(\partial_s)|_{s=0} = - \partial_t f_0 A_t f_0^{-1} 
\end{align*}
and
\begin{align*}
    & (F^* \Theta)(\partial_s, \partial_t) |_{s=0} = \bar{T}_{f_t}(f_t^{-1} v_t, f_t^{-1} (\dot \omega + u_{t, \ptr_t})) \\
    & = f_0^{-1} T_{\ptr_t}(v_t, \dot \omega + u_{t, \ptr_t}) \\
    & = (F^* d\theta)(\partial_s, \partial_t)|_{s=0} + (F^* \eta)(\partial_s) (F^* \theta)(\partial_t)|_{s=0} \\
    & \qquad - (F^* \eta)(\partial_t) (F^* \theta)(\partial_s)|_{s=0} \\
    & = f_0^{-1} \dot k_t - f_0^{-1} \dot v_t + f_0^{-1} A_t ( \dot \omega_t + u_{t,\ptr_t}).
\end{align*}
It follows that we have equations $v_0 =0$, $A_0 =0$
\begin{align*}
    \dot A_t & = R_{\ptr_t}( \dot \omega + u_{t, \ptr_t}, v_t), \\
    \dot k_t & = \dot v_t - T_{\ptr_t}(\dot \omega + u_{t, \ptr_t}, v_t) - A_t( \dot \omega_t + u_{t,\ptr_t}),
\end{align*}
which is the statement of Lemma~\ref{lemma:AffineDev}.
\end{proof}

\bibliographystyle{abbrv}
\bibliography{Bibliography}

\end{document}